\documentclass[11pt,a4paper]{amsart}
\usepackage[babel]{csquotes}
\usepackage{enumitem}
\usepackage{amsmath,amsthm,amssymb,mathrsfs,amsfonts,verbatim,enumitem,color,leftidx,mathtools}
\usepackage{mathabx}
\usepackage{etoolbox} 
\usepackage{bbm}
\usepackage[all,tips]{xy}
\usepackage{graphicx,ifpdf}
\usepackage{stmaryrd}
\usepackage{amssymb}
\usepackage{dsfont}

\usepackage{color}
\definecolor{red}{rgb}{1,0,0}

\definecolor{gre}{rgb}{0,0.7,0}

\definecolor{blu}{rgb}{0,0,1}

\ifpdf
\fi
\usepackage[colorlinks]{hyperref}
\hypersetup{
linkcolor=blue,          
citecolor=green,        
}

\usepackage{tikz}

\newtheorem{thm}{Theorem}[section]
\newtheorem{lem}[thm]{Lemma}
\newtheorem{cor}[thm]{Corollary}
\newtheorem{prop}[thm]{Proposition}

\theoremstyle{definition}

\theoremstyle{remark}

\numberwithin{equation}{section}
\definecolor{esperance}{rgb}{0.0,0.5,0.0}

\newcommand{\bb}{\mathbf{b}}

\newcommand{\bd}{\mathbf{d}}
\newcommand{\be}{\mathbf{e}}

\newcommand{\bm}{\mathbf{m}}

\newcommand{\bp}{\mathbf{p}}

\newcommand{\bu}{\mathbf{u}}
\newcommand{\bv}{\mathbf{v}}
\newcommand{\bw}{\mathbf{w}}
\newcommand{\bx}{\mathbf{x}}
\newcommand{\by}{\mathbf{y}}

\newcommand{\balpha}{\boldsymbol{\alpha}}
\newcommand{\bbeta}{\boldsymbol{\beta}}
\newcommand{\bxi}{\boldsymbol{\xi}}
\newcommand{\bupsilon}{\boldsymbol{\upsilon}}

\newcommand{\R}{\mathbb{R}}

\newcommand{\supp}{\mathrm{supp\,}}

\newcommand{\del}{\delta}

\newcommand{\eps}{\epsilon}

\newcommand{\sig}{\sigma}

\newcommand{\om}{\omega}
\newcommand{\Om}{\Omega}

\newcommand{\cB}{\mathcal{B}}
\newcommand{\cC}{\mathcal{C}}
\newcommand{\cD}{\mathcal{D}}

\newcommand{\cF}{\mathcal{F}}

\newcommand{\cK}{\mathcal{K}}
\newcommand{\cL}{\mathcal{L}}

\newcommand{\cN}{\mathcal{N}}

\newcommand{\cP}{\mathcal{P}}

\newcommand{\cR}{\mathcal{R}}
\newcommand{\cS}{\mathcal{S}}
\newcommand{\cT}{\mathcal{T}}

\newcommand{\bC}{\mathbb{C}}

\newcommand{\bM}{\mathbb{M}}

\newcommand{\bR}{\mathbb{R}}
\newcommand{\bZ}{\mathbb{Z}}
\newcommand{\bQ}{\mathbb{Q}}

\newcommand{\bN}{\mathbb{N}}

\newcommand{\bT}{\mathbb{T}}
\newcommand{\bS}{\mathbb{S}}

\newcommand{\SL}{\operatorname{SL}}
\newcommand{\SO}{\operatorname{SO}}

\newcommand*{\transp}[2][-1mu]{\ensuremath{\mskip1mu\prescript{\smash{\mathrm t\mkern#1}}{}{\mathstrut#2}}}

\newcommand\set[1]{\left\{#1\right\}}

\newcommand{\onto}{\xymatrix{\ar@{>>}[r]&}}

\newcommand{\eq}[1]
{
\begin{equation*}
{#1}
\end{equation*}
}
\newcommand{\eqlabel}[2]
{
\begin{equation}
{#2}\label{#1}
\end{equation}
}

\makeatletter
\newcommand*{\rom}[1]{\expandafter\@slowromancap\romannumeral #1@}
\makeatother

\begin{document}

\title[Multi-dimensional affine lattices]{Poissonian pair correlation for directions in multi-dimensional affine lattices, and escape of mass estimates for embedded horospheres}
\author{Wooyeon Kim and Jens Marklof}
\date{25 February 2023/15 March 2024. To appear in Ergodic Theory and Dynamical Systems.}
\thanks{Research supported by EPSRC grant EP/S024948/1. W.K. is supported by Korea Foundation for Advanced Studies (KFAS)}

\begin{abstract}
We prove the convergence of moments of the number of directions of affine lattice vectors that fall into a small disc, under natural Diophantine conditions on the shift. Furthermore, we show that the pair correlation function is Poissonian for {\em any} irrational shift in dimension 3 and higher, including well-approximable vectors. Convergence in distribution was already proved in the work of Str\"ombergsson and the second author, and the principal step in the extension to convergence of moments is an escape of mass estimate for averages over embedded $\SL(d,\R)$-horospheres in the space of affine lattices.
\end{abstract}

\maketitle
\section{Introduction}

It is often difficult to rigorously determine the pseudorandom properties of a given sequence of real numbers modulo one, including even the simplest second-order correlation functions. In the present paper we consider the problem in higher dimension and construct an explicit sequence of points $\bupsilon_1,\bupsilon_2,\bupsilon_3,\ldots$ on the unit sphere $\bS^{d-1}$ whose two-point statistics converge to that of a Poisson point process. This sequence is given by the unit vectors $\bupsilon_j=\| \by_j \|^{-1} \by_j$ representing the directions of vectors $\by_j$ in a fixed affine lattice in $\bR^d$ of unit covolume. Here the $\by_j$ are listed in increasing  length $\|\by_j\|$, where $\| \,\cdot\,\|$ denotes the Euclidean norm. If there are two or more vectors of the same length we take them in arbitrary order (our results will not depend on the choice made). If there are several lattice points with the same direction, they will appear repeatedly in the sequence. 
Our approach extends the results of \cite{EMV15}, which in turn builds on \cite{MS10}, from $d=2$ to higher dimensions. 
We are furthermore able to relax the Diophantine hypotheses imposed in \cite{EMV15}.

A sequence $(\bupsilon_j)_{j=1}^\infty$ on $\bS^{d-1}$ is called {\em uniformly distributed}, if for any set $\cD\subseteq\bS^{d-1}$ with $\operatorname{vol}_{\bS^{d-1}}(\partial\cD)=0$ we have that
$$
\lim_{N\to\infty} \frac{1}{N} \#\left\{ j \leq N : \bupsilon_j \in\cD \right\} = \frac{\operatorname{vol}_{\bS^{d-1}}(\cD)}{V_{\bS^{d-1}}},
$$
where $V_{\bS^{d-1}}=\operatorname{vol}_{\bS^{d-1}}(\bS^{d-1})$. The {\em pair correlation function} of the partial sequence $(\bupsilon_j)_{j=1}^N$ is defined as
$$
R^2_N(s) = \frac{1}{N} \#\left\{ (j_1,j_2) \,:\, j_1,j_2\leq N, \, j_1\neq j_2,\, c_d N^{\frac{1}{d-1}}\bd_{\bS^{d-1}}(\bupsilon_{j_1},\bupsilon_{j_2}) \leq s \right\} ,
$$ 
where $c_d= V_{\bS^{d-1}}^{-\frac{1}{d-1}}$ and $\bd_{\bS^{d-1}}$ is the standard geodesic distance for the unit sphere $\bS^{d-1}$. The scaling by $c_d N^{\frac{1}{d-1}}$ ensures we are measuring correlations in units where the mean density of points is one (note that the scaled sphere $c_d N^{\frac{1}{d-1}}\bS^{d-1}$ has volume $N$). The function $R^2_N(s)$ is known as Ripley's $K$-function in the statistical literature. 

We say the pair correlation of the sequence $(\bupsilon_j)_{j=1}^\infty$ is {\em Poissonian}, if for any $s>0$
\eqlabel{eqPPC}{\lim_{N\to \infty }R^2_N(s)=\frac{\pi^{\frac{d-1}{2}} s^{d-1}}{\Gamma(\tfrac{d+1}{2})},}
which is the volume of a ball in $\bR^{d-1}$ of radius $s$. This limit holds for example almost surely for a sequence of independent and uniformly distributed random points on $\bS^{d-1}$. It coincides with the pair correlation function of a Poisson point process in $\bR^{d-1}$ of intensity one, hence the term ``Poissoninan''.

Every affine lattice of unit covolume can be explicitly written as $\cL_{\bxi}=(\bZ^d+\bxi)M_0$,
where $\bxi\in\bR^d$ and $M_0\in G=\SL(d,\bR)$. For integer shift $\bxi\in\bZ^d$ we obtain the underlying lattice $\cL=\bZ^d M_0$. It follows from classical asymptotics for the number of affine lattice points in expanding sectors with a fixed opening angle that the sequence of directions is uniformly distributed on $\bS^{d-1}$, for any shift $\bxi\in\bR^d$. 

It is an interesting observation that a Poissonian pair correlation implies uniform distribution on general compact manifolds \cite{M20}. This fact was first proved in the case of $\bS^{1}$ by Aistleitner, Lachmann, and Pausinger \cite{ALP18}, and independently by Larcher and Stockinger \cite{LS20}. For the convergence of the pair correlation function, we will however, unlike the case of uniform distribution, require Diophantine conditions on the lattice shift $\bxi$. For $\kappa\ge d$, we say that $\bxi\in\bR^d$ is Diophantine of type $\kappa$ if there exists $C_\kappa>0$ such that
$$|\bxi\cdot\bm|_\bZ> C_\kappa|\bm|^{-\kappa}$$
for any $\bm\in\bZ^d\setminus\set{0}$, where $|\cdot|$ denotes the supremum norm of $\bR^d$, and $|\cdot|_{\bZ}$ denotes the supremum distance from $0\in\bT^d$. It is known that Lebesgue-almost all $\bxi\in\bR^d$ are of type $\kappa$ for any $\kappa>d$. We will in fact only require a milder Diophantine condition.  Define the function $ \zeta:\bR^d \times\bR_{>0}\to\bN$ by
\eqlabel{zetadef}{\zeta(\bxi,T):=\min\set{N\in\bN: \displaystyle\min_{\substack{\bm\in\bZ^d\setminus\set{0}\\ 0<|\bm|\leq N}}|\bxi\cdot\bm|_\bZ\leq \frac{1}{T}}.}
In view of Dirichlet's pigeon hole principle, we have that $\zeta(\bxi,T)\leq T^{1/d}$ and, if $\bxi$ is of Diophantine type $\kappa\geq d$, then $\zeta(\bxi,T)> (C_\kappa T)^{\frac{1}{\kappa}}$.


We say $\bxi\in\bR^d$ is {\em $(\rho,\mu,\nu)$-vaguely} Diophantine, if
$$\displaystyle\sum_{l=1}^{\infty}l^\rho\, 2^{\mu}\zeta(\bxi,2^{l-1})^{-\nu}<\infty .$$

Thus, if $\bxi$ is Diophantine type $\kappa$, then it is also $(\rho,\mu,\nu)$-vaguely Diophantine for $\kappa\mu<\nu$. 
If $\bxi$ satisfies the weaker Brjuno Diophantine condition \cite{LDG19,BF19}, then it is $(\rho,0,\nu)$-vaguely Diophantine for $0\leq\rho<\nu-1$
(see Appendix \ref{sec: Brjuno}).

\begin{thm}\label{thm000}
Let $d\geq 2$ and $\bxi\in\bR^d\setminus\bQ^d$; furthermore if $d=2$ assume that $\bxi$ is $(0,0,2)$-vaguely Diophantine. Then the pair correlation function of the sequence $(\bupsilon_j)_{j=1}^\infty$ of directions is Poissonian.
\end{thm}

We note that the hypothesis on $\bxi$ (in the case $d=2$) is satisfied for all Brjuno vectors, and thus in particular for Diophantine vectors of any type. In Appendix \ref{sec: counterexamples} we prove that there is a set of second Baire category of $\bxi\in\bR^2\setminus\bQ^2$ for which the pair correlation function diverges. This shows that the Diophantine condition is indeed required.

The pair correlation function also exists for $\bxi\in\bQ^d$ if $d\geq 3$ and is closely related to the two-point statistics of multi-dimensional Farey sequences \cite{BZ05,M13} and visible lattice points \cite{BCZ00,MS10}. The deeper reason why we see a Poisson pair correlation for $\bxi\notin\bQ^d$ is that the limit distribution is expressed through the Haar measure on the semi-direct product group $\SL(d,\bR) \ltimes \bR^d$, where the averages over double-lattice sums reduces to Siegel's mean value formula; cf.~Proposition \ref{Sievar} and \cite[Proposition 14]{EMV15}. In the case of $\bxi\in\bQ^d$, we need to apply Rogers' formulas which exhibit non-trivial correlations in the lattice sums, which explains the non-Poissonian correlations in this case. This is also the reason why three-point and higher-order correlation functions for directions in affine lattices with $\bxi\notin\bQ^d$ are non-Poissonian. Fine-scale statistics of directions have also been studied in the context of quasicrystals \cite{BGHJ14,MS15,H22}.

It is worth highlighting that in the analogous problem of directions in hyperbolic lattices, the pair correlation statistics are not Poissonian; see \cite{MV18} and references therein.

Theorem \ref{thm000} provides an example of a {\em deterministic} sequence in higher dimension whose pair correlation density is Poissonian. Other local statistics, however, deviate from the Poisson distribution in other statistical tests, as shown in \cite{MS10}. By {\em deterministic} we mean here that convergence is proved not just almost surely or in probability, but for a fixed, explicit sequence. An interesting non-Poisson random point process with Poissonian pair correlation is discussed in \cite{BS84}.

Theorem \ref{thm000} generalises results of \cite{EMV15} for two-dimensional affine lattices and under a stronger Diophantine condition on $\bxi$, as well as earlier work by Boca and Zaharescu \cite{BZ06} which was limited to almost every $\bxi\in\bR^2$. Other examples of deterministic sequences with Poissonian pair correlation in one dimension include $\sqrt n \bmod 1$ (excluding $n$ that are perfect squares) \cite{EM04,EMV15b} and the recent paper by Lutsko, Sourmelidis and Technau \cite{LST21} on $\alpha n^\theta\bmod 1$ which holds for every $\alpha>0$ and $\theta\leq 1/3$. Sequences such as $\alpha n^2\bmod 1$ \cite{RS98} or $\alpha2^n\bmod 1$ \cite{RZ99} have Poissonian pair correlation for almost every $\alpha$, but with no explicit instances of $\alpha$ currently known. For more references on recent developments on metric pair correlation problems, we refer the reader to \cite{AEM21,LS20b} and references therein.

Finally we mention work of Bourgain, Rudnick and Sarnak \cite{BRS16,BRS17}, who considered the fine-scale statistics of lattice points (without a shift) on large spheres, rather than radially projected points as in our setting. Remarkably, in dimension two, Kurlberg and Lester have recently been able to prove that all correlation functions converge to Poisson along density-one subsequences of eligible radii \cite{KL22}.

The next section will recall the convergence in distribution for the directions in affine lattices from \cite{MS10}, and then state an extension to convergence of mixed moments (Theorem \ref{mainthm}), which is the main result of this paper. An application of the Siegel mean value formula gives explicit expressions for all second-order statistics, and in particular shows that the pair correlation function is Poissonian (Corollaries \ref{cor1} and \ref{cor2}). These results thus immediately imply Theorem \ref{thm000}. Section \ref{secSpace} introduces the space of affine lattices. In Section \ref{secEscape} we prove escape-of-mass estimates for spherical averages that allow us to pass from convergence in distribution to convergence of moments. Sections \ref{mainlemma} and \ref{secProof} supply the proofs of our Main Lemma, which immediately implies Theorem \ref{mainthm}, and Corollaries \ref{cor1} and \ref{cor2}, respectively.

\subsection*{Acknowledgements} We thank Jo\~ ao Lopes Dias, Andreas Str\"ombergsson and the anonymous referee for helpful comments.

\section{Limit distribution and higher moments}\label{secLimit}

We consider the set $\cP_T$ of affine lattice points $y\in\cL_{\bxi}$ inside the ball of radius $\cB_T^d$ or, more generally, $\cP_{c,T}$ the lattice points in
the spherical shell
$$\cB_{c,T}^d:=\set{\mathbf{x}\in\bR^d: cT\leq \|\mathbf{x}\|\leq T},\quad 0\leq c<1.$$
The well known asymptotics for the number of lattice points in a large ball yields for $T\to \infty$,
$$
\#\cP_{c,T} \sim \operatorname{vol}_{\bR^d}(\cB_{c,1}^d) T^d = \frac{1-c^d}{d} V_{\bS^{d-1}} T^d .
$$

For $\sigma>0$ and $\bupsilon\in \bS^{d-1}$, we define $\mathfrak{D}_{c,T}(\sigma,\bupsilon)\subseteq \bS^{d-1}$ to be the open disc with center $\bupsilon$ and volume
$$\operatorname{vol}_{\bS^{d-1}}(\mathfrak{D}_{c,T}(\sigma,\bupsilon))=\frac{\sigma d}{1-c^d}T^{-d}.$$
Then the radius of $\mathfrak{D}_{c,T}(\sigma,\bupsilon)$ is $\asymp T^{-\frac{d}{d-1}}$, and for $T\to \infty$,
$$
\frac{\operatorname{vol}_{\bS^{d-1}}(\mathfrak{D}_{c,T}(\sigma,\bupsilon))}{V_{\bS^{d-1}}} \sim \frac{\sigma}{\#\cP_{c,T}}.
$$
Thus $\sigma$ measures the disc's volume in terms of the average density of points on the sphere; this scale is compatible with the one introduced above for the pair correlation function.

We define the counting function
$$\cN_{c,T}(\sigma, \bupsilon):=\#\set{\mathbf{y}\in\cP_{c,T}: \|\mathbf{y}\|^{-1} \mathbf{y}\in\mathfrak{D}_{c,T}(\sigma,\bupsilon)}$$
for the number of affine lattice points whose direction is contained in $\mathfrak{D}_{c,T}(\sigma,\bupsilon)$.
Note that on average over $\bupsilon$, uniform distribution implies (cf.~\cite[Section 2.3]{MS10}) that for any Borel probability measure $\lambda$ on $\bS^{d-1}$ with continuous density, we have
\eqlabel{mean}{
\lim_{T\to\infty} \int_{\bS^{d-1}} \cN_{c,T}(\sigma, \bupsilon) \lambda(d\bupsilon) =\sigma .}
This says that the expected number of affine lattice points, with direction contained in $\mathfrak{D}_{c,T}(\sigma,\bupsilon)$ for random $\bupsilon$, is $\sigma$.

We recall the following result from \cite{MS10}, which provides the full limit distribution of $\cN_{c,T}(\sigma, \bupsilon)$ with random $\bupsilon$ distributed according to a general Borel probability measure $\lambda$.

\begin{thm}\cite{MS10}\label{dist}
For $\underline{\sig}=(\sig_1,\ldots,\sig_m)\in\bR_{>0}^m$, there is a probability distribution $E_{c,\bxi}(\cdot, \underline{\sig})$ on $\bZ^m_{\geq 0}$ such that, for any $\underline{r}=(r_1,\ldots,r_m)\in\bZ^m_{\geq 0}$ and any Borel probability measure $\lambda$ on $\bS^{d-1}$, absolutely continuous with respect to Lebesgue,
$$\lim_{T\to\infty} \lambda({\bupsilon\in\bS^{d-1}: \cN_{c,T}(\sig_1,\bupsilon) = r_1,\ldots,\cN_{c,T}(\sig_m,\bupsilon) = r_m}) = E_{c,\bxi}(\underline{r},\underline{\sig}).$$
\end{thm}

The limit distribution satisfies the following properties, { cf.~Section \ref{sec:properties}:}

\begin{enumerate}
\item[(a)] $E_{c,\bxi}(\underline{r},\underline{\sig})$ is independent of $\lambda$ and $\cL$.
\item[(b)] $\displaystyle\sum_{\underline{r}\in\bZ_{\ge0}^m}r_jE_{c,\bxi}(\underline{r},\underline{\sig})=\displaystyle\sum_{r=0}^{\infty}rE_{c,\bxi}(r,\sig_j)=\sig_j$ for any $j\leq m$.
\item[(c)] For $\bxi\in\bQ^d$, $\sum_{\underline{r}\in\bZ_{\ge 0}^m} \|\underline{r}\|^s E_{c,\bxi}(\underline{r},\underline{\sig}) < \infty$ for $0\le s<d$, and $=\infty$ for $s\ge d$. 
\item[(d)] For $\bxi\notin\bQ^d$, $E_{c,\bxi}(\underline{r},\underline{\sig})=:E_{c}(\underline{r},\underline{\sig})$ is independent of $\bxi$.
\item[(e)] For $\bxi\notin\bQ^d$, $\sum_{\underline{r}\in\bZ_{\ge 0}^m} \|\underline{r}\|^s E_{c}(\underline{r},\underline{\sig}) < \infty$ for $0\le s<d+1$, and $=\infty$ for $s\ge d+1$. 
\end{enumerate}

A key ingredient of the proof of Theorem \ref{dist} is Ratner's measure classification theorem, which allows one to prove equidistribution of horospheres embedded in the space of affine lattices. An effective version of this statement was established only recently \cite{K21}.

Let us now turn to the main outcome of the present investigation, which extends the results of \cite{EMV15} to arbitrary dimension. For $\sig_1,\ldots,\sig_m>0$, $\lambda$ a Borel probability measure on $\bS^{d-1}$, and $\underline{z}=(z_1,\ldots,z_m)\in\bC^m$ let
\eqlabel{mixdef}{\bM_\lambda(T,\underline{z}):=\int_{\bS^{d-1}}(\cN_{c,T}(\sig_1,\bupsilon)+1)^{z_1}\cdots(\cN_{c,T}(\sig_m,\bupsilon)+1)^{z_m}\lambda(d\bupsilon).}
We denote the positive real part of $z\in\bC$ by $\operatorname{Re}_+(z):=\max\set{\operatorname{Re}(z),0}$.

The following is the principal theorem of this paper.

\begin{thm}\label{mainthm}
Let $\sig_1,\ldots,\sig_m>0$, and $\lambda$ a Borel probability measure on $\bS^{d-1}$ with continuous density. Choose $\bxi\in\bR^d$ and $\underline{z}=(z_1,\ldots,z_m)\in\bC^m$, such that one of the following hypotheses holds:
\begin{enumerate}
\item[{\rm (A1)}] $\operatorname{Re}_+(z_1)+\cdots+\operatorname{Re}_+(z_m)<d.$
\item[{\rm (A2)}]  $\eta:=\operatorname{Re}_+(z_1)+\cdots+\operatorname{Re}_+(z_m)<d+1$ and $\bxi$ is $(0,\eta-2,2)$-vaguely Diophantine if $d=2$ and $(d-1,\eta-d,1)$-vaguely Diophantine if $d\geq 3$.
\end{enumerate}
Then
\eqlabel{mixconv}{\lim_{T\to\infty}\bM_\lambda(T,\underline{z})=\displaystyle\sum_{\underline{r}\in\bZ_{\ge0}^m}(r_1+1)^{z_1}\cdots(r_m+1)^{z_m}E_{c,\bxi}(\underline{r},\underline{\sig}).}
\end{thm}
 
We note that if $\bxi$ is Diophantine of type $\kappa$ then under (A2) we have $\eta<2+\frac2\kappa$ if $d=2$ and $\eta<d+\frac1\kappa$ if $d\geq 3$. Thus in particular for badly approximable $\bxi$ (where $\kappa=d$) we have $\eta<3$ if $d=2$ and $\eta<d+\frac{1}{d}$ if $d\geq 3$.
 

We define the restricted moments to explain the key step of the proof of Theorem \ref{mainthm}:
\eqlabel{resmix}{\bM_\lambda^{(K)}(T,\underline{z}):=\int_{\max_j \cN_{c,T}(\sig_j,\bupsilon)\leq K}(\cN_{c,T}(\sig_1,\bupsilon)+1)^{z_1}\cdots(\cN_{c,T}(\sig_m,\bupsilon)+1)^{z_m}\lambda(d\bupsilon).}
Then Theorem \ref{dist} implies that, for any $K>0$,
\eqlabel{resmixconv}{\lim_{T\to\infty}\bM_\lambda^{(K)}(T,\underline{z})=\displaystyle\sum_{\underline{r}\in\bZ_{\ge0}^m, |\underline{r}|\leq K}(r_1+1)^{z_1}\cdots(r_m+1)^{z_m}E_{c,\bxi}(\underline{r},\underline{\sig}),}
where $|\underline{r}|:=\displaystyle\max_{1\leq j\leq m}r_m$. Thus, for the proof of Theorem \ref{mainthm} it remains to show that under (A1)--(A2),
\eqlabel{mixdiff}{\lim_{K\to\infty}\limsup_{T\to\infty}\left|\bM_\lambda(T,\underline{z})-\bM_\lambda^{(K)}(T,\underline{z})\right|=0.}
We will prove this in Section \ref{mainlemma}.

The following corollaries of Theorem \ref{mainthm} state that in particular the second moment and pair correlation converge and are Poisonnian.

\begin{cor}\label{cor1}
Let $\lambda$ be as in Theorem \ref{mainthm}. Let $d\geq 2$ and $\bxi\in\bR^d\setminus\bQ^d$; furthermore if $d=2$ assume that $\bxi$ is $(0,0,2)$-vaguely Diophantine.  Then, for any $\sig_1,\sig_2>0$, 
\eqlabel{mixmom}{\lim_{T\to\infty}\int_{\bS^{d-1}}\cN_{c,T}(\sig_1,\bupsilon)\cN_{c,T}(\sig_2,\bupsilon)\lambda(d\bupsilon)=\sig_1\sig_2+\min\set{\sig_1,\sig_2}.}
\end{cor}

Let $N=N_c(T)$ be the number of points in $\cP_{c,T}$, and let
$\bupsilon_{j}=\|\by_j\|^{-1}\by_j\in\bS^{d-1}$ be the directions of the vectors $\by_j\in\cP_{c,T}$, with $j=1,\ldots,N_c(T)$. For $f\in \operatorname{C}_0(\bS^{d-1}\times\bS^{d-1}\times\bR)$ (continuous, real-valued, and with compact support), we define the two-point correlation function
\eqlabel{tcordef}{R_N^2(f)=\frac{1}{N}\displaystyle\sum_{\substack{j_1,j_2=1\\ j_1\neq j_2}}^N f(\bupsilon_{j_1},\bupsilon_{j_2}, c_d N^{\frac{1}{d-1}}\bd_{\bS^{d-1}}(\bupsilon_{j_1},\bupsilon_{j_2})).}

\begin{cor}\label{cor2}
Let $d\geq 2$, $0\leq c<1$ and $\bxi\in\bR^d\setminus\bQ^d$; furthermore if $d=2$ assume that $\bxi$ is $(0,0,2)$-vaguely Diophantine.  Then for any $f\in \operatorname{C}_0(\bS^{d-1}\times\bS^{d-1}\times\bR)$
\eqlabel{tcor}{\lim_{T\to\infty}R_{N_c(T)}^2(f)= \frac{V_{\bS^{d-2}}}{V_{\bS^{d-1}}} \int_{\bS^{d-1}\times\bR_{\geq 0}}f(\bupsilon,\bupsilon,s)\, d\bupsilon \, s^{d-2} ds .}
\end{cor}

Corollary \ref{cor2} implies Theorem \ref{thm000} by approximating the characteristic function from above/below by $\operatorname{C}_0$ functions. The additional dependence of $f(\bupsilon_1,\bupsilon_2,s)$ on $\bupsilon_1,\bupsilon_2\in\bS^{d-1}$ can be used to generalise  Theorem \ref{thm000} to pair counting where $\bupsilon_{j_1}$ and $\bupsilon_{j_2}$ are restricted to different subsets of $\cD_1,\cD_2\subseteq\bS^{d-1}$. Set
\begin{multline*}
R^2_N(\cD_1,\cD_2,s) = \frac{1}{N} \#\bigg\{ (j_1,j_2) \,:\, j_1,j_2\leq N, \, j_1\neq j_2, \\
\bupsilon_{j_1}\in\cD_1,\bupsilon_{j_2}\in\cD_2,\,\, c_d N^{\frac{1}{d-1}}\bd_{\bS^{d-1}}(\bupsilon_{j_1},\bupsilon_{j_2}) \leq s \bigg\} .
\end{multline*}
We then have the following.

\begin{cor}\label{cor3}
Let $d\geq 2$, $0\leq c<1$ and $\bxi\in\bR^d\setminus\bQ^d$; furthermore if $d=2$ assume that $\bxi$ is $(0,0,2)$-vaguely Diophantine.  Then for any $\cD_1,\cD_2\subseteq\bS^{d-1}$ with $\operatorname{vol}_{\bS^{d-1}}(\partial\cD_1)=\operatorname{vol}_{\bS^{d-1}}(\partial\cD_2)=0$ and $s>0$, we have that
\eqlabel{tcor'}{\lim_{T\to\infty}R_{N_c(T)}^2(\cD_1,\cD_2,s)= \frac{\pi^{\frac{d-1}{2}} s^{d-1}}{\Gamma(\tfrac{d+1}{2})} \;\frac{\operatorname{vol}_{\bS^{d-1}}(\cD_1\cap\cD_2)}{V_{\bS^{d-1}}}.}
\end{cor}

\section{The space of affine lattices}\label{secSpace}
Let $G=\operatorname{SL}(d,\bR)$ and $\Gamma=\operatorname{SL}(d,\bZ)$. Define $G'=G\ltimes \bR^d$ by
$$(M,\bb)(M',\bb')=(MM', \bb M'+\bb'),$$
and let $\Gamma'=\Gamma\ltimes \bZ^d$ denote the corresponding arithmetic subgroup. The right action of $g=(M,\bb)\in G'$ on $\bR^d$ is defined by $\bx g:=\bx M +\bb$. We embed $G$ in $G'$ via the homomorphism $M\mapsto (M,0)$. In the following we will identify $G$ with the corresponding subgroup in $G'$ and use the shorthand $M$ for $(M,0)$.

Given $\sigma>0$ and $0\leq c<1$, define the cone
\eqlabel{conedef}{\mathfrak{C}_c(\sig):=\set{(x_1,x')\in\bR\times\bR^{d-1}: c<x<1,\|(1-c^d)x'\|<\sig}.}
For $g\in G'$ and any bounded set $\mathfrak{C}\subset\bR^d,$
\eqlabel{Ndef}{\cN(g,\mathfrak{C}):=\#(\mathfrak{C}\cap\bZ^dg).}
By construction, we can view $\cN(\cdot,\mathfrak{C})$ as a function on the space of affine lattices, $\Gamma'\backslash G'$. For $\by=(y_2,\ldots,y_{d})\in\bR^{d-1}$ and $t\ge0$, let
\eqlabel{yphidef}{\widetilde{n}(\by):=\left(\begin{matrix} 1 & y_2 & \cdots & y_{d}\\ & 1 & &  \\ & & \ddots &  \\ & & & 1 \end{matrix}\right),\quad \Phi_t:=\left(\begin{matrix} e^{-\frac{d-1}{d}t} &  &  & \\ & e^{\frac{t}{d}} & &  \\ & & \ddots & \\ & & & e^{\frac{t}{d}} \end{matrix}\right).}

 Set $\be_1=(1,0,\ldots,0)$. As in \cite[p.~1968]{MS10}, we define a smooth map $k:\bS^{d-1}\setminus\set{-\be_1}\to \SO(d)$ by
\eqlabel{kupdef+}{k(\bupsilon):=\textrm{exp}\left(\begin{matrix} 0 & -\by(\bupsilon) \\ \transp{\by(\bupsilon)} & 0_{d-1}\end{matrix}\right)\in \SO(d)}
with $\by(\be_1)=0$ and, for $\bupsilon=(\upsilon_1,\ldots,\upsilon_d)\in\bS^{d-1}\setminus\set{\be_1,-\be_1}$,
$$\by(\bupsilon)=\frac{\operatorname{arccos} v_1}{\sqrt{1-\upsilon_1^2}}(\upsilon_2,\ldots,\upsilon_d)\in\bR^{d-1}.$$
Note that $\|\by(\bupsilon)\|<\pi$. By construction, $\bupsilon=(\cos\|\by(\bupsilon)\|,\sin\|\by(\bupsilon)\|\frac{\by(\bupsilon)}{\|\by(\bupsilon)\|})$, and hence $\be_1 =\bupsilon k(\bupsilon)$ for all $\bupsilon\in\bS^{d-1}\setminus\set{-\be_1}$.

By an elementary geometric argument, given $\sig>0$ and $\eps>0$, there exists $T_0>0$ such that for all $\bupsilon\in\bS^{d-1}\setminus\set{-\be_1}$, $\bxi\in\bR^d$, $M_0\in G$ and $T=e^{\frac{t}{d}}\ge T_0$,
\eqlabel{geoest+}{\cN_{c,T}(\sig,\bupsilon)\leq\cN((1,\bxi)M_0k(\bupsilon)\Phi_t,\mathfrak{C}_0(\sig+\eps)).}
The argument is the same as in the two-dimensional case discussed in \cite{EMV15}; see in particular Fig.~3 (the yellow and red domains should now be viewed as higher-dimensional cones with symmetry axis along $\be_1$).

For $$\bu=(u_{12},\ldots,u_{1d},u_{23},\ldots,u_{(d-1)d})\in\bR^{\frac{d(d-1)}{2}}$$ and 
$$\bv=(v_1,v_2,\ldots,v_d) \in \cT:=\set{(v_1,\ldots,v_d)\in\bR_{>0}^d, v_1\cdots v_d=1},$$ 
let
\eqlabel{uvdef}{n(\bu):=\left(\begin{matrix} 1 & u_{12} & \cdots & u_{1d}\\ & \ddots & & \vdots \\ & & 1 & u_{(d-1)d} \\ & & & 1 \end{matrix}\right),\quad a(\bv):=\left(\begin{matrix} v_1 &  &  & \\ & v_2 & &  \\ & & \ddots & \\ & & & v_d \end{matrix}\right).}
The Iwasawa decomposition of $M\in G$ is given by
\eqlabel{Iwa}{M=n(\bu)a(\bv)k,}
where $\bu\in\bR^{\frac{d(d-1)}{2}}$, $\bv\in\cT$ and $k\in \SO(d)$. 

Consider the Siegel set
\eqlabel{Siegel}{\cS:=\set{n(\bu)a(\bv)k: k\in \SO(d), 0<v_{j+1}\leq\frac{2}{\sqrt{3}}v_{j} 
,\, \bu\in[-\frac{1}{2},\frac{1}{2}]^{\frac{d(d-1)}{2}}}.}
This set has the property that it contains a fundamental domain of $G$ and can be covered with a finite number of fundamental domains. Throughout this paper, we fix a fundamental domain of $G$ contained in $\cS$, and denote it by $\cF$. For $x\in \Gamma\backslash G$, there exists a unique $M\in\cF$ such that $x=\Gamma M$. Define $\iota:\Gamma\backslash G\to\cF$ so that $\iota(\Gamma M)=M$.

We extend the above to define a fundamental domain $\cF'$ and Siegel set $\cS'$ of the $\Gamma'$ action on $G'$ by
$$
\cF'= \{ (1,\bb)(M,0) : \bb\in [-\tfrac12,\tfrac12)^d,\, M\in\cF\} ,
$$
$$
\cS'= \{ (1,\bb)(M,0) : \bb\in [-\tfrac12,\tfrac12]^d,\, M\in\cS\} .
$$
As before, we define the map  $\iota:\Gamma'\backslash G'\to\cF'$ by $\iota(\Gamma' g)=g$.

Given $M\in G$, we define $\bv(M)$ as the $\bv$ coordinate of the Iwasawa decomposition
\eqlabel{Iwa'}{\iota(\Gamma M)=n(\bu)a(\bv)k.}
Similarly, for $g\in G'$, we define $\bv(g)$ and $\bb(g)$  as the $\bv$ and $\bb$ coordinates in 
\eqlabel{Iwa234}{\iota(\Gamma' g)=(1,\bb)(n(\bu)a(\bv)k,0).}

We also define 
\eqlabel{rdef}{r(\mathfrak{C}):=\max\set{\del_d,\sup\set{\|\bx\|:\bx\in\mathfrak{C}}}, \qquad \del_d=d4^d.}
and
\eqlabel{srdef}{s_r(g):=\max\set{1\leq i\leq d-1: v_i(g)>2c_dr},}
where $c_d=d\left(\frac{2}{\sqrt{3}}\right)^d$.

\begin{lem}\label{three-one}
For any bounded $\mathfrak{C}\subset\bR^d$, $g\in G'$, $\eta>0$,
\eqlabel{Nest'}{\cN(g,\mathfrak{C})^\eta\leq (C_dr^d)^\eta\prod_{i=1}^{s}\left(v_i^\eta\#\big([-c_drv_i^{-1},c_drv_i^{-1}]\cap(\bZ+b_i)\big)\right),}
where $\bv=\bv(g)=(v_1,\ldots,v_d)$, $\bb=\bb(g)=(b_1,\ldots,b_d)$, $r=r(\mathfrak{C})$, $s=s_r(g)$ and $C_d=4c_d$.
\end{lem}
\begin{proof}
Let $\mathfrak{D}_r$ be the smallest closed ball of radius $r$ centered at $0$ which contains $\mathfrak{C}$. Then
\eqlabel{Nest1}{
\begin{aligned}
\cN(g,\mathfrak{C})&\leq\cN(g,\mathfrak{D}_r)=\cN(\iota(\Gamma g),\mathfrak{D}_r)\\
&=\#(\mathfrak{D}_r\cap (\bZ^d+\bb)n(\bu)a(\bv))\\
&\leq \#([-r,r]^d\cap(\bZ^d+\bb)n(\bu)a(\bv))\\
&= \#\left(\prod_{i=1}^{d}[-rv_i^{-1},rv_i^{-1}]\cap(\bZ^d+\bb)n(\bu)\right).\end{aligned}}
Since $0<v_{i+1}\leq\frac{2}{\sqrt{3}}v_i$ for $1\leq i\leq d-1$, we have $v_j^{-1}\leq\left(\frac{2}{\sqrt{3}}\right)^{j-i}v_i^{-1}$ for any $1\leq i< j\leq d$. It follows that

\eqlabel{Nest2}{\begin{aligned}
\cN(g,\mathfrak{C})&\leq \#\left(\prod_{i=1}^{d}[-c_drv_i^{-1},c_drv_i^{-1}]\cap(\bZ^d+\bb)\right)\\
&\leq\prod_{i=1}^{d}\#\left([-c_drv_i^{-1},c_drv_i^{-1}]\cap(\bZ+b_i)\right)\\
&\leq\prod_{i=s+1}^{d}(2c_drv_i^{-1}+1)\times\prod_{i=1}^{s}\#([-c_drv_i^{-1},c_drv_i^{-1}]\cap(\bZ+b_i)).
\end{aligned}}
For $i\ge s+1$ we have $v_i^{-1}\ge\frac{1}{2c_dr}$ by definition of $s$. It follows that
$$\prod_{i=s+1}^{d}(2c_drv_i^{-1}+1)\leq (4c_dr)^{d-s}\prod_{i=s+1}^{d}v_i^{-1}=(4c_dr)^{d-s}\prod_{i=1}^{s}v_i,$$
hence
$$\cN(g,\mathfrak{C})\leq C_dr^d\prod_{i=1}^{s}\left(v_i\#([-c_drv_i^{-1},c_drv_i^{-1}]\cap(\bZ+b_i))\right).$$
From the fact that 
\eqlabel{zero-one}{\#([-c_drv_i^{-1},c_drv_i^{-1}]\cap(\bZ+b_i))\in\set{0,1}} for any $1\leq i\leq s$,
\eqref{Nest'} follows.
\end{proof}
The case of mixed moment will be dealt with by the inequality
\eqlabel{union}{\big| \cN(g,\mathfrak{C}_1)^{z_1}\cdots\cN(g,\mathfrak{C}_m)^{z_m}\big| \leq \big| \cN(g,\mathfrak{C}_1\cup\cdots\cup\mathfrak{C}_m)^{z_1+\cdots+z_m}\big|.}

\section{Properties of the limiting distribution}\label{sec:properties}
In this section we prove the properties (a)-(e) of the limiting distribution in Theorem \ref{dist}. We denote by $m_{X'}$, $m_{X_1}$, and $m_{X_q}$ the Haar probability measures on the homogeneous spaces
\eq{X'=\Gamma'\backslash G',\quad X_1=\Gamma\backslash G,\quad X_q=\Gamma_q\backslash G,}
respectively. Here $\Gamma_q$ denotes the congruence subgroup
\eq{\Gamma_q:=\set{\gamma\in\Gamma_q: \gamma\equiv \operatorname{id}  ( \operatorname{mod} q)}}
for $q\ge 2$. According to \cite[Theorems 6.3, 6.5 and subsequent remarks, and Lemma 9.5]{MS10}, the limiting distribution $E_{c,\bxi}(\cdot,\underline{\sig})$ in Theorem \ref{dist} is given as follows:

\eqlabel{limdist}{E_{c,\bxi}(\underline{r},\underline{\sig})= \begin{cases}
m_{X_1}\big(\set{\Gamma M\in X_1: \forall j, \; \#\big(\bZ^dM\cap \mathfrak{C}_c(\sig_j)\big)=r_j}\big) & \text{if $\bxi\in\bZ^d$},\\
m_{X_q}\big(\set{\Gamma_q M\in X_q: \forall j, \; \#\big((\bZ^d+\frac{\bp}{q})M\cap \mathfrak{C}_c(\sig_j)\big)=r_j }\big) & \text{if $\bxi=\frac{\bp}{q}\in\bQ^d\setminus \bZ^d$},\\
m_{X_q}\big(\set{\Gamma'g\in X': \forall j, \; \#\big(\bZ^d g\cap \mathfrak{C}_c(\sig_j)\big)=r_j }\big) & \text{if $\bxi\in\bR^d\setminus \bQ^d$},
\end{cases}}
where $\mathfrak{C}_c(\sig)$ is as in \eqref{conedef}.

Property (a) follows from the observation that the distribution described in \eqref{limdist} is independent of $\lambda$ and $\cL$.

For $\bxi\in\bR^d\setminus\bQ^d$ the distribution is also independent of $\bxi$, so property (d) follows.

Property (b) follows from \eqref{mean}.

Property (c) and (e) follows from calculations of \cite{M00}.
We write $g=(1,\mathbf{b})(M,0)\in G'$ with $M=n(\bu)a(\bv)k\in\cS$ as in \eqref{Iwa} and \eqref{Siegel}. For $s\in\set{1,\ldots,d-1}$, put
\eq{\cS_s:=\set{M=n(\bu)a(\bv)k\in\cS: v_{s+1}\leq 1 \leq\frac{2}{\sqrt{3}}v_s},}
and for $s=0,d,$
\eq{\cS_0:=\set{M=n(\bu)a(\bv)k\in\cS: v_1 \leq 1},}
\eq{\cS_d:=\set{M=n(\bu)a(\bv)k\in\cS: v_d\geq \frac{\sqrt{3}}{2},}}
then the sets $\cS_0$ and $\cS_d$ are clearly compact, and we also have $\cS=\displaystyle\bigcup_{s=0}^{d}\cS_s$ (see \cite[Lemma 3.12]{M00}). 
For $k\in \operatorname{SO}(d)$ we denote by $\chi_k$ the characteristic function of the set $\mathfrak{C}_c(\sig)k^{-1}$ and define
\eq{\phi_s(k,w_1,\ldots,w_s):=\int_{
\bR^{d-s}}\chi_k(w_1,\ldots,w_s,t_{s+1},\ldots,t_d)dt_{s+1}\cdots dt_{d},}
\eq{\phi_{s,\operatorname{max}}(w_1,
\ldots,w_s):=\max_{k\in\operatorname{SO}(d)}\phi_s(k,w_1,\ldots,w_s).}
We have
\eq{\cN(g,\mathfrak{C}_c(\sig))=\displaystyle\sum_{\bm\in\bZ^d}\chi_k(v_1w_1(\bm),\ldots,v_dw_d(\bm)),}
where $w_i(\bm)=(m_i+b_i)+\displaystyle\sum_{j=1}^{i-1}u_{ji}(m_j+b_j)$. Without loss of generality we may assume $b_1,\ldots,b_d\in[-\frac{1}{2},\frac{1}{2}]$. For $M\in \cS_s$ with sufficiently large $v_1\cdots v_s$,
\eq{\begin{aligned}\cN(g,\mathfrak{C}_c(\sig))&=\displaystyle\sum_{\bm\in \set{\mathbf{0}}\times\bZ^{d-s}}\chi_k(v_1w_1(\bm),\ldots,v_dw_d(\bm))\\
&=v_{s+1}^{-1}\cdots v_{d}^{-1}\phi_s\left(k,v_1x_1,\ldots,v_sx_s\right)+O(1),\end{aligned}}
where $x_i=b_i+\displaystyle\sum_{j=1}^{i-1}u_{ji}b_j$ for $i=1,\ldots,s$.

We first consider the case of $E_{c,\bxi}$ for $\bxi\in\bZ^d$. For $r_0\to\infty$ and $m_G$ denoting the Haar measure of $G$ (with arbitrary normalisation),
\eq{\begin{aligned}
\displaystyle\sum_{r=r_0}^{\infty}E_{c,\bxi}(r,\sig)&=m_{X_1}\left(\set{M\in X_1: \cN(g,\mathfrak{C}_c(\sig))\ge r_0}\right)
\\&\asymp\displaystyle\sum_{s=0}^{d}m_G\left(\{M\in\cS_s: v_1\cdots v_s\phi_s(k,0,\ldots,0)\ge r_0+O(1)\}\right)
\\&\asymp\displaystyle\sum_{s=0}^{d}m_G\left(\left\{M\in\cS_s: v_1\cdots v_s\ge \frac{r_0+O(1)}{\phi_{s,\operatorname{max}}(0,\ldots,0)}\right\}\right). \end{aligned}}
In the last line we are using the continuity of $\phi_{s,\operatorname{max}}$ with respect to $k\in\operatorname{SO}(d)$. According to the calculation of \cite[Proof of Theorem 3.11]{M00} with $n=2$, the sum in the last line is $\asymp r_0^{-d}$. This proves property (c) for $\bxi\in\bZ^d$. The case of other $\bxi\in\bQ^d$ is analogous. 

In the case of $\bxi\in\bR^d\setminus\bQ^d$ we get
\eq{\begin{aligned}
\displaystyle\sum_{r=r_0}^{\infty}E_{c,\bxi}(r,\sig)&=m_{X'}\left(\set{g\in X': \cN(g,\mathfrak{C}_c(\sig))\ge r_0}\right)
\\&\asymp\displaystyle\sum_{s=0}^{d}m_G\left(\{M\in\cS_s: v_1\cdots v_s\phi_s(k,v_1x_1,\ldots,v_sx_s)\ge r_0+O(1)\}\right)
\\&\asymp\displaystyle\sum_{s=0}^{d}m_G\left(\left\{M\in\cS_s: v_1\cdots v_s\ge \frac{r_0+O(1)}{\phi_{s,\operatorname{max}}(v_1x_1,\ldots,v_sx_s)}\right\}\right). \end{aligned}}
In this case we use the calculation of \cite[Proof of Theorem 4.3]{M00} which implies that the sum in the last line is $\asymp r_0^{-d-1}$. This proves property (e).
\section{Escape of mass}\label{secEscape}

Denote by $\chi_{I}$ the characteristic function of a subset $I\subseteq\bR$. 
For $R\ge 1$ and $\eta,r>0$, define the $\Gamma'$-invariant function $F_{R,\eta,r}:G'\to\bR$ by
\eqlabel{Fetadef}{\begin{aligned}
F_{R,\eta,r}(g):=\chi_{[R,\infty)}\bigg(\prod_{i=1}^{s_r(g)}v_i(g)\bigg)\; \prod_{i=1}^{s_r(g)} v_i(g)^\eta \chi_{[-c_dr,c_dr]}\big(v_i(g) b_i(g)\big) .
\end{aligned}}
In view of Lemma \ref{three-one}, \eqref{zero-one} and \eqref{union}, we note that for 
\eqlabel{etar}{\eta=\operatorname{Re}(z_1)+\cdots+\operatorname{Re}(z_m),\quad r=r(\mathfrak{C}_1\cup\cdots\cup\mathfrak{C}_m),}
and all $g\in G$ such that $\prod_{i=1}^{s_r(g)}v_i(g)\ge R$ with $R$ sufficiently large, we have that
\eqlabel{NFrel}{\big| \cN(g,\mathfrak{C}_1)^{z_1}\cdots\cN(g,\mathfrak{C}_m)^{z_m}\big|\leq (C_dr^d)^\eta F_{R,\eta,r}(g).}

The following proposition establishes under which conditions there is no escape of mass in the equidistribution of horospheres with respect to the function $F_{R,\eta,r}$ and thus also for $\cN(g,\mathfrak{C}_1)^{z_1}\cdots\cN(g,\mathfrak{C}_m)^{z_m}$.

\begin{prop}\label{mainprop}
Let $\bxi\in\bR^d$, $M_0\in G$, $\eta,r>0$, and $\psi\in \operatorname{C}_0(\bR^{d-1})$. Assume that one of the following hypotheses hold:
\begin{enumerate}
\item[{\rm (B1)}] $\eta<d.$
\item[{\rm (B2)}] $\eta<d+1$ and $\bxi$ is $(0,\eta-2,2)$-vaguely Diophantine if $d=2$ and $(d-1,\eta-d,1)$-vaguely Diophantine if $d\geq 3$.
\end{enumerate}
Then
\eqlabel{nonescape}{\lim_{R\to\infty}\limsup_{t\to \infty}\left|\int_{\by\in\bR^{d-1}}F_{R,\eta,r}\left(\Gamma'(1,\bxi)M_0\widetilde{n}(\by)\Phi_t\right)\psi(\by)d\by\right|=0.}
\end{prop}

To prepare for the proof of this statement, put $\cK:=\supp \psi$. Without loss of generality we may assume $\cK\subset[-1,1]^{d-1}$. Indeed, there exists $s_0\ge 0$ such that $e^{-s_0}\cK\subset[-1,1]^{d-1}$, so we may replace $M_0,\by$, and $\Phi_t$ in \eqref{nonescape} by $M_0\Phi_{-s_0}, e^{s_0}\by$, and $\Phi_{t+s_0}$, respectively, and reduce it to the case $\cK\subset[-1,1]^{d-1}$.

Next we define two maps $\gamma=\gamma_t:\bR^{d-1}\to\Gamma$ and $h=h_t:\bR^{d-1}\to\cF$ as follows. For $\by\in\bR^{d-1}$, $t\in\bR$, there exist unique $\gamma(\by)=\gamma_t(\by)\in\Gamma$ and $h(\by)=h_t(\by)\in\cF$ such that
$$M_0\widetilde{n}(\by)\Phi_t=\gamma(\by)h(\by).$$
Note that $\Gamma'(1,\bxi)M_0\widetilde{n}(\by)\Phi_t$ in \eqref{nonescape} can now be expressed as
$$\Gamma'(1,\bxi)M_0\widetilde{n}(\by)\Phi_t=\Gamma'(1,\bxi\gamma(\by))(h(\by),0).$$
For $1\leq s\leq d-1$ and $\underline{l}=(l_1,\ldots,l_s)\in\bZ^s_{\ge0}$, we let
\eqlabel{Xilsdef}{\Xi^s_{\underline{l}}:=\set{g\in\cF:s_r(g)=s, \del_d 2^{l_i}<  v_i(g)\leq\del_d 2^{l_i+1} (i=1,\ldots,s)}} with $\del_d=d4^d$. Then for $g=(1,\bxi \gamma(\by))h(\by)$ with $h(\by)\in\Xi_{\underline{l}}^s$ we have
$$
F_{R,\eta,r}(\Gamma'(1,\bxi)M_0\widetilde{n}(\by)\Phi_t)=F_{R,\eta,r}(\Gamma'(1,\bxi \gamma(\by))h(\by))
\leq \delta_d^\eta2^{\eta l}\chi_{[R,\infty)}\big(\delta_d^\eta2^{\eta l}\big),$$
where $l=l_1+\cdots+l_s$. It follows that the integral in \eqref{nonescape} is bounded by
\eqlabel{Xidecomp}{
\begin{aligned}
&\left|\int_{\by\in\cK}F_{R,\eta,r}(\Gamma'(1,\bxi)M_0\widetilde{n}(\by)\Phi_t)d\by\right|\\
&\ll\displaystyle\sum_{s=1}^{d-1}\displaystyle\sum_{l=\lfloor\log_2 R\rfloor}^{\infty} 2^{\eta l} \displaystyle\sum_{\substack{\underline{l}=(l_1,\ldots,l_s)\in\bZ^s_{\ge0}, \\ l_1+\cdots+l_s=l}}\operatorname{vol}_{\bR^{d-1}}\left(\set{\by\in\cK: h(\by)\in\Xi_{\underline{l}}^s} \right).\end{aligned}}

This will be sufficient for proving case (B1). For (B2) we need a refinement that also considers the size of $\bxi\gamma(\by)$; see \eqref{Xidecomp'} below.

Let us write 
$\bbeta_i(\by):=\be_i \transp{h(\by)^{-1}}$ for $1\leq i\leq d$ and $\by\in\bR^{d-1}$, and consider the Iwasawa decomposition of $h(\by)$, $$h(\by)=n(\bu(\by))a(\bv(\by))k(\by).$$

\begin{lem}\label{betaprod}
If $h(\by)\in\Xi^s_{\underline{l}}$ for $\underline{l}\in\bZ_{\ge0}^{s}$, then $|\bbeta_i(\by)|< 2^{-l_i}$ for all $1\leq i\leq s$.
\end{lem}
\begin{proof}
For the sake of simplicity we write $v_{i}=v_{i}(h(\by))$ for $1\leq i\leq d$ and $u_{ij}=u_{ij}(h(\by))$ for $1\leq i<j\leq d$. We also define $\widetilde{\bu}(\by)=\left(\widetilde{u}_{ij}\right)_{1\leq i<j\leq d}$ by $n(\widetilde{\bu}(\by))=n(\bu(\by))^{-1}$. Note that each $\widetilde{u}_{ij}$ can be expressed in terms of at most $2^d$ monomials of $u_{12},\ldots,u_{(d-1)d}$ with coefficients $\pm 1$, hence $|\widetilde{u}_{ij}|\leq 2^d$ for any $1\leq i<j\leq d$.

If $h(\by)\in\Xi^s_{l}$, then we have
\eq{\begin{aligned}\bbeta_i(\by)&=\be_i \transp{n(\bu(\by))}^{-1}\transp{a(\bv(\by))}^{-1}\transp{k}^{-1}\\
&=\left(\be_i-\sum_{j=1}^{i-1}\widetilde{u}_{ji}\be_j\right)a(\bv(\by))^{-1}k\\
&=\left(v_i^{-1}\be_i-\sum_{j=1}^{i-1}v_{j}^{-1}\widetilde{u}_{ji}\be_j\right)k
\end{aligned}}
for $1\leq i\leq d$. Since $v_{j}^{-1}\leq \left(\frac{2}{\sqrt{3}}\right)^{i-j}v_{i}^{-1}\leq 2^dv_i^{-1}$ for any $1\leq j<i\leq d$,
\eq{|\bbeta_i(\by)|=\left|v_i^{-1}\be_i-\sum_{j=1}^{i-1}v_{j}^{-1}\widetilde{u}_{ji}\be_j\right|\leq d4^d v_i^{-1}=\del_d v_i^{-1}< 2^{-l_i}}
for all $1\leq i\leq s$.
\end{proof}

Denote by $\pi_1:\bR^d\to\bR$ and the orthogonal projection to the first coordinate and $\pi':\bR^d\to\bR^{d-1}$ the orthogonal projection to the remaining $(d-1)$ coordinates. Let $\Lambda=\bZ^d\transp{M_0^{-1}}$ and, for $k\in\bZ$, let
\eqlabel{lambdak}{\cR_k:=\set{\bx\in\bR^d: |\pi_1\bx|<2^{k+2}, 2^{k}\leq |\pi'\bx|<2^{k+1}},\quad \Lambda_k:=\cR_k\cap \Lambda.}
Then for sufficiently large $k$
\eqlabel{Lambdakcount}{\#\Lambda_k\asymp 2^{kd},}
where the implied constants are independent of $k$ but depend on the fixed $M_0\in G$. Throughout this section, let $K_0\in\bZ$ be the largest integer such that
$$\set{\bx\in\Lambda: |\pi_1\bx|<2^{k+1}, |\pi'\bx|<2^{k}}=\emptyset$$
for all $k\leq K_0$.
Note that $K_0$ only depends on the choice of $\Lambda$.

We define the norm $\| \;\cdot\; \|$ for the wedge product by
$$
\| \bx_1 \wedge \cdots \wedge \bx_n \|^2 = \det\left[( \bx_i \transp{\bx_j})_{ij}\right].
$$
For $\underline{k}=(k_1,\ldots,k_s)\in\bZ_{\ge K_0}^s$ and $\underline{p}=(p_1,\ldots,p_s)\in\bZ_{\ge0}^s$, we denote by $\Lambda_{\underline{k}}(\underline{p})$ the set of $(\bx_1,\ldots,\bx_s)\in\Lambda_{k_1}\times\cdots\times\Lambda_{k_s}$ such that
\eqlabel{perpsize}{0<2^{-p_j-1}\|\bx_j\| \left\|\bigwedge_{i=1}^{j-1}\bx_i\right\|<\left\|\bigwedge_{i=1}^{j}\bx_i\right\|\leq 2^{-p_j}\|\bx_j\| \left\|\bigwedge_{i=1}^{j-1}\bx_i\right\|}
for $j=1,\ldots,s$. Then any $(\bx_1,\ldots,\bx_s)\in\Lambda_{k_1}\times\cdots\times\Lambda_{k_s}$ such that $\bx_1,\ldots,\bx_s$ are $\bR$-linearly independent is contained in $\bigcup_{\underline{p}\in\bZ^s_{\ge0}}\Lambda_{\underline{k}}(\underline{p})$.

\begin{lem}\label{pcount}
For any $\underline{k}\in\bZ^s_{\ge K_0}$ and $\underline{p}\in\bZ^s_{\ge0}$, 
$$\#\Lambda_{\underline{k}}(\underline{p})\ll \prod_{j=1}^s 2^{\om_s(k_j,p_j)},$$
where 
$$\om_s(k,p)= \begin{cases}
dk-(d+1-s)p & \text{if $p\leq k$}\\
sk-p & \text{if $p>k$.}
\end{cases} 
$$
Moreover, if there exists $1\leq j\leq s$ such that $p_j\ge jk_j+K$, then $\#\Lambda_{\underline{k}}(\underline{p})=0$, where $K$ is a sufficiently large constant depending only on the choice of lattice $\Lambda$.
\end{lem}
\begin{proof}
Given $\bx_1,\ldots,\bx_{j-1}$, let $V$ be the subspace spanned by $\bx_1,\ldots,\bx_{j-1}$ and denote by $\Upsilon$ the region of $\bx_j$ satisfying 
$$\left\|\bigwedge_{i=1}^{j}\bx_i\right\|\leq 2^{-p_j}\|\bx_j\| \left\|\bigwedge_{i=1}^{j-1}\bx_i\right\|.$$ Note that $\Upsilon\cap \cR_k$ has width $\asymp 2^{k_j}$ along the directions in $V$, and width $\asymp 2^{k_j-p_j}$ along the directions perpendicular to $V$. 
 
 If $p_j\leq k_j$, then the number of possible $\bx_j\in\Lambda_{k_j}$ satisfying \eqref{perpsize} is therefore at most $\ll (2^{k_j})^{j-1}(2^{k_j-p_j})^{d-(j-1)}=2^{dk_j-(d+1-j)p_j}$. 
 
In case $p_j> k_j$, let $j'$ be the maximal number of $\bR$-linearly independent vectors in $\Upsilon\cap\Lambda_{k_j}$. We may assume $j\leq j'\leq d$ since there is no $\bx_j$ satisfying \eqref{perpsize} in $\Upsilon$ otherwise. Then we can take a $j'$-dimensional parallelepiped $\mathcal{Q}$ generated by $\widebar{\bx}_1,\ldots,\widebar{\bx}_{j'}\in\Upsilon\cap\Lambda_{k_j}$ such that there is no element of $\Upsilon\cap\Lambda_{k_j}$ inside $\mathcal{Q}$. Since $\widebar{\bx}_1,\ldots,\widebar{\bx}_{j'}$ are $\bR$-linearly independent and contained in $\Lambda=\bZ^d\transp{M_0^{-1}}$, the $j'$-dimensional volume of $\mathcal{Q}\transp{M_0}$ is $\ge 1$. Hence, the $j'$-dimensional volume of $\mathcal{Q}$ is $\gg 1$ independently of $p_j$ and $\bx_1,\ldots,\bx_{j-1}$. Also, the interior of the sets $\bx_j+\mathcal{Q}$ with $\bx_j\in\Lambda_{k_j}$ are pairwise disjoint. Note that $\mathcal{Q}$ is contained in $j'(\Upsilon\cap\cR_{k_j})$ since the generators are in $\Upsilon\cap\cR_{k_j}$. Thus for any $\bx_j\in\Upsilon\cap\Lambda_{k_j}$, the set $\bx_j+\mathcal{Q}$ is contained in $(j'+1)(\Upsilon\cap\cR_{k_j})$ and the $j'$-dimensional volume of this region is 
$$\ll (j'+1)^{j'}(2^{k_j})^{j-1}(2^{k_j-p_j})^{j'-(j-1)}\ll 2^{jk_j-p_j}.$$ Because of this and the uniform lower bound on the volume of $\mathcal{Q}$, it follows that the number of possible $\bx_j\in\Lambda_{k_j}$ satisfying \eqref{perpsize} is at most $\ll 2^{jk_j-p_j}$. In particular, there is no such $\bx_j\in\Lambda_{k_j}$ if $p_j\ge jk_j+K$.
 
We have shown that for fixed $\bx_1,\ldots,\bx_{j-1}$, the number of possible $\bx_j\in\Lambda_{k_j}$ satisfying \eqref{perpsize} is $\ll 2^{\om_j(k_j,p_j)}\leq 2^{\om_s(k_j,p_j)}$. Hence the desired estimate follows.
\end{proof}

For $\underline{l}=(l_1,\ldots,l_s)\in\bZ_{\ge0}^s$ and $(\bx_1,\ldots,\bx_s)\in\Lambda^s$, let $\Om_{\underline{l}}(\bx_1,\ldots,\bx_s)$ be the set of $\by\in\cK=\supp\psi\subset[-1,1]^{d-1}$ satisfying the following two conditions:
\begin{itemize}
    \item $\be_i\transp{\gamma}(\by)\transp{M_0^{-1}}=\bx_i$ for $i=1,\ldots,s$,
    \item $|\bbeta_i(\by)|< 2^{-l_i}$ for $i=1,\ldots,s$.
\end{itemize}

\begin{lem}\label{msrbdd}
There exists $T_0\ge0$ such that the following holds for any $t>T_0$. For $\underline{l}\in\bZ_{\ge0}^s$ and $(\bx_1,\ldots,\bx_s)\in\Lambda^s$, the set $\Om_{\underline{l}}(\bx_1,\ldots,\bx_s)$ is the empty set if there exists $1\leq i\leq s$ such that $$\bx_i\notin \displaystyle\bigcup_{k=K_0}^{\lfloor\frac{t}{d\log 2}-l_i\rfloor}\Lambda_k.$$ If ${K_0\leq }k_i\leq \lfloor\frac{t}{d\log 2}-l_i\rfloor$ for all $1\leq i\leq s$ and $(\bx_1,\ldots,\bx_s)\in\Lambda_{\underline{k}}(
 \underline{p})$, then
\eqlabel{ptestim}{\operatorname{vol}_{\bR^{d-1}}(\Om_{\underline{l}}(\bx_1,\ldots,\bx_s))\ll e^{-\frac{(d-1)s}{d}t} \prod_{i=1}^s 2^{p_i-l_i-k_i}.}
\end{lem}
\begin{proof}
For $\by\in\Om_{\underline{l}}(\bx_1,\ldots,\bx_s)$, by definition of $\bbeta_i(\by)$ we have
$$|\bx_i\transp{\widetilde{n}(-\by)}\Phi_{-t}|<2^{-l_i}$$
for $i=1,\ldots,s$. By a straightforward computation with $\bx_i=(\pi_1\bx_i,\pi'\bx_i)\in\bR\times\bR^{d-1}$, it implies that
\eqlabel{eqn1}{|\pi_1\bx_i-\pi'\bx_i\cdot\by|<2^{-l_i}e^{-\frac{d-1}{d}t},}
\eqlabel{eqn2}{|\pi'\bx_i|<2^{-l_i}e^{\frac{t}{d}}}
for $i=1,\ldots,s$.

If there exists $i$ such that $\bx_i\in\Lambda_{k_i}$ with $k_i>\frac{t}{d\log 2}-l_i$, then it contradicts \eqref{eqn2}. If there exists $i$ such that $\bx_i\notin \displaystyle\bigcup_{k=K_0}^{\infty}\Lambda_k$, then we have $|\pi_1\bx_i|>2|\pi'\bx_i|$, which also contradicts \eqref{eqn1} and \eqref{eqn2} since they imply
$$|\pi_1\bx_i|\leq |\pi'\bx_i||\by|+2^{-l_i}e^{-\frac{d-1}{d}t}< |\pi'\bx_i|+2^{K_0-1}.$$
This proves the first claim of the lemma.

Suppose now that $k_i\leq \frac{t}{d\log 2}-l_i$ for all $i=1,\ldots,s$ and $(\bx_1,\ldots,\bx_s)\in\Lambda_{\underline{k}}(\underline{p})$. To prove the estimate \eqref{ptestim}, we may assume $\Om_{\underline{l}}(\bx_1,\ldots,\bx_s)\neq\emptyset$ and pick any $\by\in\Om_{\underline{l}}(\bx_1,\ldots,\bx_s)$. Let $p_1',\ldots,p_s'\in\bZ_{\ge0}$ be the integers such that
$$2^{-p_j'-1}\left\|\bigwedge_{i=1}^{j-1} \frac{\pi'\bx_i}{\|\pi'\bx_i\|}\right\|< \left\|\bigwedge_{i=1}^j \frac{\pi'\bx_i}{\|\pi'\bx_i\|}\right\|\leq2^{-p_j'}\left\|\bigwedge_{i=1}^{j-1} \frac{\pi'\bx_i}{\|\pi'\bx_i\|}\right\|$$
for $j=1,\ldots,s$. 

We will prove that $p_j'\leq p_j+O(1)$ for all $1\leq j\leq s$. Since $\frac{\pi'\bx_i}{\|\pi'\bx_i\|}$'s are unit vectors, for each $j$ we can find $c_1,\ldots,c_{j-1}\ll 1$ such that
\eqlabel{vecdiff'}{\left\|\frac{\pi'\bx_j}{\|\pi'\bx_j\|}-\displaystyle\sum_{i=1}^{j-1}c_i\frac{\pi'\bx_i}{\|\pi'\bx_i\|}\right\|\ll 2^{-p_j'}.}
This in turn implies that, for any choice of $\by$,
\eqlabel{vecdiff''}{\left|\frac{\pi'\bx_j\cdot \by}{\|\pi'\bx_j\|}-\displaystyle\sum_{i=1}^{j-1}c_i\frac{\pi'\bx_i\cdot\by}{\|\pi'\bx_i\|}\right|\ll 2^{-p_j'} \|\by\| .} 
Now, if $\by\in\Om_{\underline{l}}(\bx_1,\ldots,\bx_s)$, then we have \eqref{eqn1}, and
using the triangle inequality and $\|\pi'\bx_i\|\gg 2^{k_i}$, it follows that
\eqlabel{vecdiff1}{\left|\frac{\pi_1\bx_j}{\|\pi'\bx_j\|}-\displaystyle\sum_{i=1}^{j-1}c_i\frac{\pi_1\bx_i}{\|\pi'\bx_i\|}\right|\ll 2^{-p_j'}+\displaystyle\sum_{i=1}^{j}2^{-l_i-k_i}e^{-\frac{d-1}{d}t}.}
Combining \eqref{vecdiff'} and \eqref{vecdiff1}, we get
\eqlabel{vecdiff}{\left\|\frac{\bx_j}{\|\pi'\bx_j\|}-\displaystyle\sum_{i=1}^{j-1}c_i\frac{\bx_i}{\|\pi'\bx_i\|}\right\|\ll 2^{-p_j'}+\displaystyle\sum_{i=1}^{j}2^{-l_i-k_i}e^{-\frac{d-1}{d}t},}
hence
\eqlabel{wedgedec}{\left\|\bigwedge_{i=1}^j \frac{\bx_i}{\|\pi'\bx_i\|}\right\|\ll\left(2^{-p_j'}+\displaystyle\sum_{i=1}^{j}2^{-l_i-k_i}e^{-\frac{d-1}{d}t}\right)\left\|\bigwedge_{i=1}^{j-1} \frac{\bx_i}{\|\pi'\bx_i\|}\right\|}
for $j=1,\ldots,s$. Since $\|\pi'\bx_i\|\asymp\|\bx_i\|$ for all $i=1,\ldots,s$ by definition of $\Lambda_{k_i}$'s, we can replace the $\|\pi'\bx_i\|$'s in \eqref{wedgedec} with $\|\bx_i\|$'s for $i=1,\ldots,j$. By \eqref{perpsize} it implies that 
\eqlabel{pjpj'}{2^{-p_j}\leq C\left(2^{-p_j'}+\displaystyle\sum_{i=1}^{j}2^{-l_i-k_i}e^{-\frac{d-1}{d}t}\right)}
for some constant $C>0$. 

On the other hand, we have
$$\prod_{i=1}^{j}2^{k_i-p_i}\gg\prod_{i=1}^{j}2^{-p_i}\prod_{i=1}^{j}\|\bx_j\|\asymp \left\|\bigwedge_{i=1}^{j}\bx_i\right\|\ge 1,$$
hence $p_j\leq \sum_{i=1}^{j}p_i\leq \sum_{i=1}^{j}k_i+O(1)$ for any $1\leq j\leq s$.
 It follows that $2^{-p_j}\gg 2^{-k_i}e^{-\frac{j-1}{d}t}\geq 2^{-k_i}e^{-\frac{d-2}{d}t}$ for all $1\leq i\leq j$, so for $C$ as above,
 \eqlabel{pjcond}{2^{-p_j} \ge 2C\displaystyle\sum_{i=1}^{j}2^{-l_i-k_i}e^{-\frac{d-1}{d}t}}
 holds for sufficiently large $t$. Combining with \eqref{pjpj'}, we have $p_j'\leq p_j+O(1)$ for all $1\leq j\leq s$.

 For each $i$, the set of $\by\in\cK$ satisfying \eqref{eqn1} is a $2^{-l_i}e^{-\frac{d-1}{d}t}\|\pi'\bx_i\|^{-1}$-thickened hyperplane in $\bR^{d-1}$ which is perpendicular to $\pi'\bx_i$. Therefore, $\operatorname{vol}_{\bR^{d-1}}(\Om_{\underline{l}}(\bx_1,\ldots,\bx_s))$ is the volume of the intersection of such $s$-number of hyperplanes and the compact set $\cK$. The intersection has width $\ll2^{-l_i}e^{-\frac{d-1}{d}t}\|\pi'\bx_i\|^{-1}$ along the direction of $\pi'\bx_i$ for $1\leq i\leq s$. It follows that the volume of the intersection is bounded above by
\eq{\begin{aligned}
&\ll \left\|\bigwedge_{i=1}^s \frac{\pi'\bx_i}{\|\pi'\bx_i\|}\right\|^{-1}\prod_{i=1}^s (2^{-l_i}e^{-\frac{d-1}{d}t}\|\pi'\bx_i\|^{-1})\\
&\ll \left(\prod_{i=1}^{s}2^{-p_i'}\right)^{-1}\prod_{i=1}^s (2^{-l_i}e^{-\frac{d-1}{d}t}\|\bx_i\|^{-1})\ll e^{-\frac{(d-1)s}{d}t}\prod_{i=1}^{s}2^{p_i-l_i-k_i}.
\end{aligned}}

\end{proof}
\begin{proof}[Proof of Proposition \ref{mainprop} under (B1)]
By Lemma \ref{betaprod},
\eqlabel{betabdd}{\begin{aligned}
\operatorname{vol}_{\bR^{d-1}}\left(\set{\by\in\cK:h(\by)\in\Xi_{\underline{l}}^s}\right)\leq \operatorname{vol}_{\bR^{d-1}}\left(\set{\by\in\cK:|\bbeta_i(\by)|<2^{-l_i}(i=1,\ldots,s)}\right).
\end{aligned}}
For $\underline{l}=(l_1,\ldots,l_s)\in\bZ_{\ge0}^s$, by Lemma \ref{msrbdd} we have
\eq{\begin{aligned}
&\operatorname{vol}_{\bR^{d-1}}\left(\set{\by\in\cK: |\bbeta_i(\by)|<2^{-l_i}(i=1,\ldots,s)}\right)\\
&=\displaystyle\sum_{(\bx_1,\ldots,\bx_s)\in\Lambda^s}\operatorname{vol}_{\bR^{d-1}}(\Om_{\underline{l}}(\bx_1,\ldots,\bx_s))\\
&=\displaystyle\sum_{k_1=K_0}^{\lfloor\frac{t}{d\log 2}-l_1\rfloor}\cdots\displaystyle\sum_{k_s=K_0}^{\lfloor\frac{t}{d\log 2}-l_s\rfloor}\displaystyle\sum_{\underline{p}\in\bZ_{\ge0}^s}\displaystyle\sum_{(\bx_1,\ldots,\bx_s)\in\Lambda_{\underline{k}}(\underline{p})}\operatorname{vol}_{\bR^{d-1}}(\Om_{\underline{l}}(\bx_1,\ldots,\bx_s))\\
&\ll\displaystyle\sum_{k_1=K_0}^{\lfloor\frac{t}{d\log 2}-l_1\rfloor}\cdots\displaystyle\sum_{k_s=K_0}^{\lfloor\frac{t}{d\log 2}-l_s\rfloor}\displaystyle\sum_{\underline{p}\in\bZ_{\ge0}^s}\#\Lambda_{\underline{k}}(\underline{p})e^{-\frac{(d-1)s}{d}t}\prod_{i=1}^{s}2^{p_i-l_i-k_i}.\end{aligned}}
Using the estimate of $\#\Lambda_{\underline{k}}(\underline{p})$ in Lemma \ref{pcount}, we get

\eq{\begin{aligned}
&\ll e^{-\frac{(d-1)s}{d}t}\displaystyle\prod_{i=1}^s\left(\sum_{k_i={K_0}}^{\lfloor\frac{t}{d\log 2}-l_i\rfloor}\left(\displaystyle\sum_{p_i=0}^{k_i}2^{dk_i-(d-s+1)p_i}2^{p_i-l_i-k_i}+\displaystyle\sum_{p_i=k_i+1}^{sk_i+K}2^{sk_i-p_i}2^{p_i-l_i-k_i}\right)\right)\\
&\ll e^{-\frac{(d-1)s}{d}t}\displaystyle\prod_{i=1}^s\left(\sum_{k_i=K_0}^{\lfloor\frac{t}{d\log 2}-l_i\rfloor} \left(2^{(d-1)k_i-l_i}+(sk_i+K)2^{(s-1)k_i-l_i}\right)\right)\\ &\ll e^{-\frac{(d-1)s}{d}t}\displaystyle\prod_{i=1}^s\left(\sum_{k_i=K_0}^{\lfloor\frac{t}{d\log 2}-l_i\rfloor}2^{(d-1)k_i-l_i}\right)\ll 2^{-d\sum l_i}.
\end{aligned}}
Hence combining with \eqref{Xidecomp}, for any $t,R>0$ and $\eta<d$ we obtain
\eqlabel{measurem}{\begin{aligned}
&\left|\int_{\by\in\cK}F_{R,\eta,r}(\Gamma'(1,\bxi)M_0\widetilde{n}(\by)\Phi_t)d\by\right|\\
&\ll \displaystyle\sum_{s=1}^{d-1}\displaystyle\sum_{l=\lfloor\log_2 R\rfloor}^{\infty} l^{s-1}2^{-(d-\eta) l}\ll R^{-(d-\eta+o(1))},
\end{aligned}}
which completes the proof.
\end{proof}

\begin{proof}[Proof of Proposition \ref{mainprop} under (B2)]
Recall that $\Gamma'(1,\bxi)M_0\widetilde{n}(\by)\Phi_t$ in \eqref{nonescape} is expressed as
$$\Gamma'(1,\bxi)M_0\widetilde{n}(\by)\Phi_t=\Gamma'(1,\bxi\gamma(\by))\, h(\by).$$

For $1\leq s\leq d-1$  and $\underline{l}=(l_1,\ldots,l_s)\in\bZ_{\ge0}^s$ with $l_1+\cdots+l_s=l$, we define $\Xi^s_{\underline{l}}\subseteq\cF$ as in \eqref{Xilsdef} and $\Gamma_{\underline{l}}^{s,r}\subseteq\Gamma$ by
\eqlabel{Xixi0def}{\begin{aligned}
\Gamma_{\underline{l}}^{s,r}&:=\set{\gamma\in\Gamma: |(\bxi\gamma)\cdot \be_i|_{\bZ}\leq \del_{d,r}2^{-l_i}(i=1,\ldots,s)}
\end{aligned}}
where $\delta_{d,r}=\frac{c_d r}{\delta_d}$. 
For $g=(1,\bxi \gamma(\by))h(\by)$ with $h(\by)\in\Xi_{\underline{l}}^s$ we have
\begin{align*}
F_{R,\eta,r}(\Gamma'(1,\bxi)M_0\widetilde{n}(\by)\Phi_t)&=F_{R,\eta,r}(\Gamma'(1,\bxi \gamma(\by))h(\by))\\
&\leq \delta_d^\eta2^{\eta l}\chi_{[R,\infty)}\big(\delta_d^\eta2^{\eta l}\big)\prod_{i=1}^{s}\chi_{[-\delta_{d,r} 2^{-l_i},\delta_{d,r} 2^{-l_i}]}\big(b_i(g)\big)\end{align*}
since, by construction, we have
$v_i(g) \geq  \delta_d 2^{l_i}$ and thus $\frac{c_d r}{v_i(g)} \leq  \delta_{d,r} 2^{-l_i}$.
Since $|b_i(g)|=|(\bxi\gamma)\cdot \be_i|_{\bZ}$, the integral in \eqref{nonescape} is bounded by
\eqlabel{Xidecomp'}{
\begin{aligned}
&\left|\int_{\by\in\cK}F_{R,\eta,r}(\Gamma'(1,\bxi)M_0\widetilde{n}(\by)\Phi_t)d\by\right|\\
&\ll\displaystyle\sum_{s=1}^{d-1}\displaystyle\sum_{l=\lfloor\log_2 R\rfloor}^{\infty} 2^{\eta l}\displaystyle\sum_{\substack{\underline{l}=(l_1,\ldots,l_s)\in\bZ^s_{\ge0} \\ l_1+\cdots+l_s=l}} \operatorname{vol}_{\bR^{d-1}}\left(\set{\by\in\cK: h(\by)\in\Xi_{\underline{l}}^s, \gamma(\by)\in \Gamma_{\underline{l}}^{s,r}} \right).\end{aligned}}
This is the required improvement on \eqref{Xidecomp}.

The remaining task is thus to estimate the measure of the set of $\by\in\cK$ such that $h(\by)\in\Xi_{\underline{l}}^s$ and $\gamma(\by)\in\Gamma_{\underline{l}}^{s,r}$. Recall that $\bZ^d=\Lambda \transp{M_0}$. From now on we fix $r>0$ as in \eqref{etar} and no longer record the implicit dependence of constants on this parameter. For $\underline{k}\in\bZ_{\ge K_0}^s$, $\underline{l}\in\bZ_{\ge0}^s$, and $\underline{p}\in\bZ_{\ge0}^s$, we denote by $\Lambda_{\underline{k}}^{\underline{l}}(\underline{p})$ the set of elements in $(\bx_1,\ldots,\bx_s)\in\Lambda_{\underline{k}}(\underline{p})$ satisfying
\eqlabel{xiDio}{|\bxi\cdot (\bx_i\transp{M_0})|_{\bZ}\leq \delta_{d,r}2^{-l_i}}
for all $1\leq i\leq s$. As we counted the number of lattice points of $\Lambda_{\underline{k}}(\underline{p})$ in Lemma \ref{pcount}, here we count the number of lattice points of $\Lambda_{\underline{k}}^{\underline{l}}(\underline{p})$ as follows.


\begin{lem}\label{pcount'}
For $\bxi\in\bR^d$ let
$$\om_{s,\bxi}(k,p,l)= \begin{cases}
dk-(d+1-s)p-d\log_2\frac{\zeta(\bxi,2^{l-1})}{4\|M_0\|} & \text{if $p\leq k-\log_2\frac{\zeta(\bxi,2^{l-1})}{4\|M_0\|}$,}\\
(s-1)k-(s-1)\log_2\frac{\zeta(\bxi,2^{l-1})}{4\|M_0\|} & \text{if $0\leq k-\log_2\frac{\zeta(\bxi,2^{l-1})}{4\|M_0\|}<p\leq k$,}\\
sk-p & \text{if $p>k$}.
\end{cases}
$$
For any $\underline{k}\in\bZ_{\ge K_0}^s$, $\underline{l}\in\bZ_{\ge0}^s$, and $\underline{p}\in\bZ^s_{\ge0}$,
$$\#\Lambda_{\underline{k}}^{\underline{l}}(\underline{p})\ll \prod_{j=1}^{s}2^{\om_{s,\bxi}(k_j,p_j,l_j)}.$$
Moreover, if there exists $1\leq j\leq s$ such that $p_j\ge jk_j+K$ or $k_j<\log_2\frac{\zeta(\bxi,2^{l-1})}{4\|M_0\|}$, then $\#\Lambda_{\underline{k}}(\underline{p})=0$. Here, $K$ is a sufficiently large constant depending on the choice of lattice $\Lambda$.
\end{lem}
\begin{proof}
For $(\bx_1,\ldots,\bx_s)\in\Lambda_{\underline{k}}^{\underline{l}}(p)\subseteq\Lambda_{\underline{k}}(p)$, recall that
\eqlabel{perpsize'}{\left\|\bigwedge_{i=1}^{j}\bx_i\right\|\leq 2^{-p_j}\|\bx_j\|\left\|\bigwedge_{i=1}^{j-1}\bx_i\right\|}
for $j=1,\ldots,s$. Given $\bx_1,\ldots,\bx_{j-1}$, in the proof of Lemma \ref{pcount} we already showed that the possible number of $\bx_j\in\Lambda_{k_j}$ is $\ll 2^{\om_s(k_j,p_j)}=2^{sk_j-p_j}$ if $p_j>k_j$. Hence, it is enough to show that this bound can be improved under the assumption $p_j\leq k_j$.

We first consider the case $p_j\leq k_j-\log_2\frac{\zeta(\bxi,2^{l-1})}{4\|M_0\|}$. In this case, if $\bx_j\in\Lambda_{k_j}$ satisfies \eqref{perpsize'}, then $\bx_j$ must be $\ll 2^{k_j-p_j}$-close to the subspace $V$ spanned by $\bx_1,\ldots,\bx_{j-1}$. Hence, the region of $\bx_j$ satisfying \eqref{perpsize'} has width $2^{k_j-p_j}$ along the directions perpendicular to $V$, and width $2^{k_j}$ along the directions of $V$. This region can be covered with at most $$\ll2^{dk_j-(d+1-j)p_j}\zeta(\bxi,\delta_{d,r}^{-1}2^{l_j-1})^{-d}\ll 2^{dk_j-(d+1-j)p_j}\zeta(\bxi,2^{l_j-1})^{-d}$$ cubes with sidelength $\|M_0\|^{-1}\zeta(\bxi,\delta_{d,r}^{-1}2^{l_j-1})$. 

We claim that there is at most one point of $\Lambda_{k_j}$ satisfying \eqref{xiDio} in each cube with sidelength $\|M_0\|^{-1}\zeta(\bxi,\delta_{d,r}^{-1}2^{l_j-1})$. To see this, suppose that there are two distinct points $\bx,\bx'\in\Lambda_{k_j}$ with distance $< \|M_0\|^{-1}\zeta(\bxi,\delta_{d,r}^{-1}2^{l_j-1})$ satisfying \eqref{xiDio}. Then we have $|\bxi\cdot(\bx\transp{M_0}-\bx'\transp{M_0})|_\bZ\leq \delta_{d,r}2^{-l_j+1}$
and $$0<|\bx\transp{M_0}-\bx'\transp{M_0}|\leq\|M_0\||\bx-\bx'|<\zeta(\bxi,\delta_{d,r}^{-1}2^{l_j-1}).$$
However, by definition \eqref{zetadef} there is no $\bm\in\bZ^d\setminus\set{0}$ with $|\bm|<\zeta(\bxi,\delta_{d,r}^{-1}2^{l_j-1})$ and $|\bxi\cdot\bm|_\bZ\leq \delta_{d,r}2^{-l_j+1}$. Hence the claim is proved. It follows from the claim that the number of possible $\bx_j\in\Lambda_{k_j}$ satisfying \eqref{xiDio} and \eqref{perpsize'} is at most $\ll2^{dk_j-(d+1-j)p_j}\left(\frac{\zeta(\bxi,2^{l_j-1})}{4\|M_0\|}\right)^{-d}= 2^{\om_{s,\bxi}(k_j,p_j,l_j)}$. 

We now consider the case $k_j-\log_2\frac{\zeta(\bxi,2^{l-1})}{4\|M_0\|}<p_j\leq k_j$. In this case, the region of $\bx_j$ satisfying \eqref{perpsize'} can be covered with at most 
$$\ll 2^{(j-1)k_j}\zeta(\bxi,\delta_{d,r}^{-1}2^{l_j-1})^{-(j-1)}\ll2^{(j-1)k_j}\zeta(\bxi,2^{l_j-1})^{-(j-1)}$$ cubes with sidelength $\|M_0\|^{-1}\zeta(\bxi,\delta_{d,r}^{-1}2^{l_j-1})$ since this region has width $2^{k_j-p_j}< \frac{\zeta(\bxi,\delta_{d,r}^{-1}2^{l_j-1})}{4\|M_0\|}$ along the directions perpendicular to $V$ and width $2^{k_j}$ along the directions of $V$. Similar to the previous case, the number of possible $\bx_j\in\Lambda_{k_j}$ satisfying \eqref{xiDio} and \eqref{perpsize'} is at most 
$$\ll\left(2^{k_j}\left(\frac{\zeta(\bxi,2^{l_j-1})}{4\|M_0\|}\right)^{-1}\right)^{j-1}\leq \left(2^{k_j}\left(\frac{\zeta(\bxi,2^{l_j-1})}{4\|M_0\|}\right)^{-1}\right)^{s-1}= 2^{\om_{s,\bxi}(k_j,p_j,l_j)}.$$ 

We have shown that for fixed $\bx_1,\ldots,\bx_{j-1}$, the number of possible $\bx_j\in\Lambda_{k_j}$ satisfying \eqref{xiDio} and \eqref{perpsize'} is $\ll 2^{\om_{s,\bxi}(k_j,p_j,l_j)}$. Hence we obtain the desired estimate.

As we have shown in Lemma \ref{pcount}, $\#\Lambda_{\underline{k}}(\underline{p})=0$ if there exists $1\leq j\leq s$ such that $p_j\geq jk_j+K$. On the other hand, if there exists $1\leq j\leq s$ such that $k_j<\log_2\frac{\zeta(\bxi,2^{l-1})}{4\|M_0\|}$, then we have $|\bx_j\transp{M_0}|\leq \|M_0\|2^{k_j+2}< \zeta(\bxi,2^{l-1})$ since $\bx_j\in\Lambda_{k_j}$. It follows that $|\bxi\cdot (\bx_j\transp{M_0})|>2^{-l+1}$ by definition of $\zeta(\bxi,T)$. In other words, there is no $\bx_j\in\Lambda_{k_j}$ satisfying \eqref{xiDio}, hence $\#\Lambda_{\underline{k}}(\underline{p})=0$.
\end{proof}

The rest of the argument is similar to the case (B1). 
By Lemma \ref{betaprod} and Lemma \ref{msrbdd} we have 
\eq{\begin{aligned}
&\operatorname{vol}_{\bR^{d-1}}\left(\set{\by\in\cK: h(\by)\in\Xi_{\underline{l}}^s, \gamma(\by)\in\Gamma_{\underline{l}}^{s,r}}\right)\\
&\leq \operatorname{vol}_{\bR^{d-1}}\left(\set{\by\in\cK: |\bbeta_i(\by)|<2^{-l_i}, |\bxi\cdot \be_i\transp{\gamma(\by)}|_{\bZ}< \delta_{d,r}2^{-l_i}(i=1,\ldots,s)}\right)\\
&\leq\displaystyle\sum_{k_1=K_0}^{\lfloor\frac{t}{d\log 2}-l_1\rfloor}\cdots\displaystyle\sum_{k_s=K_0}^{\lfloor\frac{t}{d\log 2}-l_s\rfloor}\displaystyle\sum_{\underline{p}\in\bZ_{\ge0}^s}\displaystyle\sum_{(\bx_1,\ldots,\bx_s)\in\Lambda_{\underline{k}}^{\underline{l}}(\underline{p})}\operatorname{vol}_{\bR^{d-1}}(\Om_{\underline{l}}(\bx_1,\ldots,\bx_s))\\
&\ll\displaystyle\sum_{k_1=K_0}^{\lfloor\frac{t}{d\log 2}-l_1\rfloor}\cdots\displaystyle\sum_{k_s=K_0}^{\lfloor\frac{t}{d\log 2}-l_s\rfloor}\sum_{\underline{p}\in\bZ_{\ge0}^s}\#\Lambda_{\underline{k}}^{\underline{l}}(\underline{p})e^{-\frac{(d-1)s}{d}t}\prod_{i=1}^{s}2^{p_i-l_i- k_i}\end{aligned}}
for $\underline{l}=(l_1,\ldots,l_s)\in\bZ_{\ge0}^s$. Using the estimate of $\#\Lambda_{\underline{k}}^{\underline{l}}(\underline{p})$ in Lemma \ref{pcount'}, we obtain

\eqlabel{Ombdd''}{\begin{aligned}
&\operatorname{vol}_{\bR^{d-1}}\left(\set{\by\in\cK: h(\by)\in\Xi_{\underline{l}}^s, \gamma(\by)\in\Gamma_{\underline{l}}^{s,r}}\right)\\
&\ll e^{-\frac{(d-1)s}{d}t}\prod_{i=1}^{s}\left(\displaystyle\sum_{k_i=K_0}^{\lfloor\frac{t}{d\log 2}-l_i\rfloor}\displaystyle\sum_{p_i=0}^{sk_i+K}2^{\om_{s,\bxi}(k_i,p_i,l_i)}2^{p_i-l_i-k_i}\right)\\
&=e^{-\frac{(d-1)s}{d}t}\prod_{i=1}^{s}\left(\displaystyle\sum_{k_i=\lceil \log_2\frac{\zeta(\bxi,2^{l_i-1})}{4\|M_0\|}\rceil}^{\lfloor\frac{t}{d\log 2}-l_i\rfloor}\displaystyle\sum_{p_i=0}^{sk_i+K}2^{\om_{s,\bxi}(k_i,p_i,l_i)}2^{p_i-l_i-k_i}\right).
\end{aligned}}

We may assume 
\eqlabel{lbound}{\frac{t}{d\log 2}-l_i>\log_2\frac{\zeta(\bxi,2^{l_i-1})}{4\|M_0\|}}
for all $i$ since otherwise the product over $i$ is zero.

We split the double sum in the last line of \eqref{Ombdd''} as follows,
\eq{\begin{aligned}
&\displaystyle\sum_{k_i=\lceil \log_2\frac{\zeta(\bxi,2^{l_i-1})}{4\|M_0\|}\rceil}^{\lfloor\frac{t}{d\log 2}-l_i\rfloor}\displaystyle\sum_{p_i=0}^{sk_i+K}2^{\om_{s,\bxi}(k_i,p_i,l_i)}2^{p_i-l_i-k_i}\\&=\displaystyle\sum_{k_i=\lceil \log_2\frac{\zeta(\bxi,2^{l_i-1})}{4\|M_0\|}\rceil}^{\lfloor\frac{t}{d\log 2}-l_i\rfloor}\left(\displaystyle\sum_{p_i=0}^{k_i-\lfloor \log_2\frac{\zeta(\bxi,2^{l_i-1})}{4\|M_0\|}\rfloor-1}2^{(d-1)k_i-(d-s)p_i-l_i-d\log_2\frac{\zeta(\bxi,2^{l_i-1})}{4\|M_0\|}}\right. \\&\left.\qquad\qquad+\displaystyle\sum_{p_i=k_i-\lfloor \log_2\frac{\zeta(\bxi,2^{l_i-1})}{4\|M_0\|}\rfloor}^{k_i}2^{(s-2)k_i+p_i-l_i-(s-1)\log_2\frac{\zeta(\bxi,2^{l_i-1})}{4\|M_0\|}}+\displaystyle\sum_{p_i=k_i+1}^{sk_i+K}2^{(s-1)k_i-l_i}\right) .
\end{aligned}}
This is bounded above by
\eq{\begin{aligned}
&\ll \displaystyle\sum_{k_i=\lceil \log_2\frac{\zeta(\bxi,2^{l_i-1})}{4\|M_0\|}\rceil}^{\lfloor\frac{t}{d\log 2}-l_i\rfloor}\big(2^{(d-1)k_i-l_i}\zeta(\bxi,2^{l_i-1})^{-d}+2^{(s-1)k_i-l_i}\zeta(\bxi,2^{l_i-1})^{-(s-1)}+(sk_i+K)2^{(s-1)k_i-l_i}\big)\\
&\ll \displaystyle\sum_{k_i=\lceil \log_2\frac{\zeta(\bxi,2^{l_i-1})}{4\|M_0\|}\rceil}^{\lfloor\frac{t}{d\log 2}-l_i\rfloor}\big(2^{(d-1)k_i-l_i}\zeta(\bxi,2^{l_i-1})^{-d}+2^{(s-1)k_i-l_i}\zeta(\bxi,2^{l_i-1})^{-(s-1)}+k_i2^{(d-2)k_i-l_i}\big)\\&\ll e^{\frac{d-1}{d}t}2^{-dl_i}\zeta(\bxi,2^{l_i-1})^{-d}+e^{\frac{s-1}{d}t}2^{-sl_i}\zeta(\bxi,2^{l_i-1})^{-(s-1)}+te^{\frac{d-2}{d}t}2^{-(d-1)l_i}.
\end{aligned}}
In summary, we have established that
\eqlabel{sum1'}{\begin{aligned}
&e^{-\frac{d-1}{d}t}\displaystyle\sum_{k_i=\lceil \log_2\frac{\zeta(\bxi,2^{l_i-1})}{4\|M_0\|}\rceil}^{\lfloor\frac{t}{d\log 2}-l_i\rfloor}\displaystyle\sum_{p_i=0}^{sk_i+K}2^{\om_{s,\bxi}(k_i,p_i,l_i)}2^{p_i-l_i-k_i}\\
&\ll 2^{-dl_i}\zeta(\bxi,2^{l_i-1})^{-d}+e^{-\frac{d-s}{d}t}2^{-sl_i}\zeta(\bxi,2^{l_i-1})^{-(s-1)}+te^{-\frac{t}{d}}2^{-(d-1)l_i}\\
&\ll  2^{-dl_i}\zeta(\bxi,2^{l_i-1})^{-d}+2^{-dl_i}\zeta(\bxi,2^{l_i-1})^{-(d-1)}+\big(l_i+\log_2\zeta(\bxi,2^{l_i-1})\big)2^{-dl_i}\zeta(\bxi,2^{l_i-1})^{-1}\\
&\ll  \big(l_i+\log_2\zeta(\bxi,2^{l_i-1})\big)2^{-dl_i}\zeta(\bxi,2^{l_i-1})^{-1} \ll l_i2^{-dl_i}\zeta(\bxi,2^{l_i-1})^{-1}.
\end{aligned}}
Here we have used that $e^{-\frac{t}{d}}< 2^{-l_i} 4\|M_0\| \zeta(\bxi,2^{l_i-1})^{-1}$ and
$$\frac{t}{d}\,e^{-\frac{t}{d}}< \big(l_i+\log_2\frac{\zeta(\bxi,2^{l_i-1})}{4\|M_0\|}\big)2^{-l_i} 4\|M_0\| \zeta(\bxi,2^{l_i-1})^{-1}$$ for $t>d$,  both
of which follow from \eqref{lbound}. (Note that $x\mapsto x 2^{-x}$ is strictly decreasing for $x>\frac{1}{\log 2}$.)

Combining \eqref{Ombdd''} and \eqref{sum1'}, we have
\eqlabel{B2volest}{\operatorname{vol}_{\bR^{d-1}}\left(\set{\by\in\cK: h(\by)\in\Xi_{\underline{l}}^s, \gamma(\by)\in\Gamma_{\underline{l}}^{s,r}}\right)\ll l^s2^{-dl}\prod_{i=1}^{s}\zeta(\bxi,2^{l_i-1})^{-1}}
for any $\underline{l}=(l_1,\ldots,l_s)\in\bZ_{\ge0}^{s}$ with $l_1+\cdots+l_s=l$. 

For any $t,R>0$ it follows from \eqref{Xidecomp'} and \eqref{B2volest} that
\eqlabel{impo1}{\begin{aligned} & \left|\int_{\by\in[-1,1]^{d-1}} F_{R,\eta,r}(\Gamma'(1,\bxi)M_0\widetilde{n}(\by)\Phi_t)d\by\right| \\
& \ll \displaystyle\sum_{s=1}^{d-1} \displaystyle\sum_{l=\lfloor\log_2 R\rfloor}^{s\lfloor\frac{t}{d\log 2}\rfloor} 
\displaystyle\sum_{\substack{\underline{l}=(l_1,\ldots,l_s)\in\bZ^s_{\ge0} \\ l_1+\cdots+l_s=l}}
l^s 2^{(\eta-d)l}\prod_{i=1}^{s}\zeta(\bxi,2^{l_i-1})^{-1} \\
& \ll \displaystyle\sum_{s=1}^{d-1}  
\displaystyle\sum_{\substack{\underline{l}=(l_1,\ldots,l_s)\in\bZ^s_{\ge0} \\  l_1+\cdots+l_s \geq \lfloor\log_2 R\rfloor }}
\prod_{i=1}^{s} (l_i+1)^s 2^{(\eta-d)l_i} \zeta(\bxi,2^{l_i-1})^{-1} .
\end{aligned}}
Since the final bound of \eqref{impo1} is independent of $t$ and
$$
\displaystyle\sum_{(l_1,\ldots,l_s)\in\bZ^s_{\ge0}}
\prod_{i=0}^{s} (l_i+1)^s 2^{(\eta-d)l_i} \zeta(\bxi,2^{l_i-1})^{-1} 
= 
\left( \sum_{l=0}^\infty (l+1)^s 2^{(\eta-d)l} \zeta(\bxi,2^{l-1})^{-1}  \right)^s <\infty 
$$
converges by assumption (B2) (note $s\leq d-1$), we can conclude that \eq{\begin{aligned}
\lim_{R\to\infty}\limsup_{t\to \infty}&\left|\int_{\by\in[-1,1]^{d-1}} F_{R,\eta,r}\left(\Gamma'(1,\bxi)M_0\widetilde{n}(\by)\Phi_t\right) d\by\right| =0
\end{aligned}}
as required.

We now discuss the case that $d=2$ and $\bxi$ is $(0,\eta-2,2)$-vaguely Diophantine. If $d=2$, then $s=1$, and in view of the definition \eqref{perpsize} the set $\Lambda_{k}(p)$ is the empty set unless $p=0$. Hence, the double sum in the last line of \eqref{Ombdd''} is written
$$\displaystyle\sum_{k=\lceil \log_2\frac{\zeta(\bxi,2^{l})}{4\|M_0\|}\rceil}^{\lfloor\frac{t}{2\log 2}-l\rfloor}2^{\om_{s,\bxi}(k,0,l)}2^{-l-k}=\displaystyle\sum_{k=\lceil \log_2\frac{\zeta(\bxi,2^{l})}{4\|M_0\|}\rceil}^{\lfloor\frac{t}{2\log 2}-l\rfloor}2^{k-l}\zeta(\bxi,2^{l-1})^{-2}$$
and bounded above by $\ll e^{\frac{t}{2}}2^{-2l}\zeta(\bxi,2^{l-1})^{-2}$. Plugging this in \eqref{Ombdd''}, we get
$$\operatorname{vol}_{\bR}\left(\set{\by\in\cK: h(\by)\in\Xi_{l}^1, \gamma(\by)\in\Gamma_{l}^{1,r}}\right)\ll 2^{-2l}\zeta(\bxi,2^{l-1})^{-2}.$$
Note that here we gained additional decay of $\zeta(\bxi,2^{l-1})^{-1}$ in comparison to \eqref{B2volest}. It follows that
$$\left|\int_{\by\in[-1,1]} F_{R,\eta,r}(\Gamma'(1,\bxi)M_0\widetilde{n}(\by)\Phi_t)d\by\right|\ll \displaystyle\sum_{l=\lfloor\log_2 R\rfloor}^{\lfloor\frac{t}{2\log 2}\rfloor} 2^{(\eta-2)l}\zeta(\bxi,2^{l-1})^{-2}<\infty.$$
\end{proof}

\begin{lem}\label{LogLipschitz}
    For any compact set $\cC\subset G$, there exists $C=C(\cC)>1$ such that $F_{R,\eta,r}\big(g(h,0)\big)\leq C^{(d-1)\eta}  F_{C^{-(d-1)}R,\eta,Cr}(g)$ for any $h\in\cC$ and $g\in G'$.
\end{lem}
\begin{proof}
Let $g=(1,\bb)(M,0)$ with $M\in G$ and $\bb\in \bR^d$. For each $1\leq i\leq d$, we have $v_i(g)=\|\be_i\bv(M)\|=\|\be_iM\|$ and $v_i\big(g(h,0)\big)=\|\be_i\bv(Mh)\|=\|\be_iMh\|$, hence there exists $C=C(\cC)>1$ such that $C^{-1}v_i\big(g(h,0)\big)\leq v_i(g)\leq Cv_i\big(g(h,0)\big)$ for any $h\in \cC$. We also have $b_i(g)=b_i\big(g(h,0)\big)$ for all $i$ since $b_i$ is invariant under $\operatorname{SL}(d,\bR)$-action. Let $g'=g(h,0)$. It follows that
\eq{\begin{aligned}
    &F_{R,\eta,r}(g')=\chi_{[R,\infty)}\bigg(\prod_{i=1}^{s_r(g')}v_i(g')\bigg)\; \prod_{i=1}^{s_r(g')} v_i(g')^\eta \chi_{[-c_dr,c_dr]}\big(v_i(g') b_i(g')\big)\\
    &\leq C^{(d-1)\eta}\chi_{[C^{-(d-1)}R,\infty)}\bigg(\prod_{i=1}^{s_r(g)}v_i(g)\bigg)\; \prod_{i=1}^{s_r(g)} v_i(g)^\eta \chi_{[-c_dCr,c_dCr]}\big(v_i(g) b_i(g)\big)\\
    &=C^{(d-1)\eta}F_{C^{-(d-1)}R,\eta,Cr}(g).
\end{aligned}}
\end{proof}

Let us denote by $\bS^{d-1}_+$ and $\bS^{d-1}_-$ the upper hemisphere and the lower hemisphere, respectively, i.e. $$\bS^{d-1}_+=\set{(\bupsilon_1,\ldots,\bupsilon_d)\in \bS^{d-1}: \bupsilon_1\ge 0},$$
$$\bS^{d-1}_-=\set{(\bupsilon_1,\ldots,\bupsilon_d)\in \bS^{d-1}: \bupsilon_1\leq 0}.$$

By the construction of the map in \eqref{kupdef+}, $k$ is smooth and its differential is non-singular and bounded on $\bS^{d-1}_+$. For $\bupsilon\in\bS^{d-1}_+$ we may write
\eqlabel{k+decomposition}{k(\bupsilon)=\widetilde{n}(\by(\bupsilon))\left(\begin{matrix} c(\bupsilon) & 0 \\ \transp{\bw(\bupsilon)} & A(\bupsilon)\end{matrix}\right)}
for $c(\bupsilon)>0$, $\bw(\bupsilon)\in\bR^{d-1}$, and $A(\bupsilon)\in \operatorname{Mat}_{d-1,d-1}(\bR)$. Note that $\by,c,\bw$, and $A$ are smooth and bounded on $\bS^{d-1}_+$.

\begin{prop}\label{mainprop'}Let $\lambda$ be a Borel probability measure on $\bS^{d-1}$ with continuous density, $\bxi\in\bR^d$ and $\eta>0$ so that (B1) or (B2) holds. Then
\eqlabel{nonescape1+}{\lim_{R\to\infty}\limsup_{t\to \infty}\left|\int_{\bupsilon\in\bS^{d-1}_+}F_{R,\eta,r}\left(\Gamma'(1,\bxi)M_0k(\bupsilon)\Phi_t\right)d\lambda(\bupsilon)\right|=0.}

\end{prop}
\begin{proof}
This follows from Proposition \ref{mainprop} by the same argument as in the proof of \cite[Corollary 5.4]{MS10}. We first observe that
\eq{\begin{aligned}
    \Gamma'(1,\bxi)M_0k(\bupsilon)\Phi_t&=\Gamma'(1,\bxi)M_0\widetilde{n}(\by(\bupsilon))\left(\begin{matrix} c(\bupsilon) & 0 \\ \transp{\bw(\bupsilon)} & A(\bupsilon)\end{matrix}\right)\Phi_t\\
    &=\Gamma'(1,\bxi)M_0\widetilde{n}(\by(\bupsilon))\Phi_t\left(\begin{matrix} c(\bupsilon) & 0 \\ e^{-t}\transp{\bw(\bupsilon)} & A(\bupsilon)\end{matrix}\right).
\end{aligned}}
Since $c,\bw$, and $A$ are bounded on $\bS^{d-1}_+$, we may choose a compact set $\cC\subset G$ so that $\left(\begin{matrix} c(\bupsilon) & 0 \\ e^{-t}\transp{\bw(\bupsilon)} & A(\bupsilon)\end{matrix}\right)\in \cC$ for any $\bupsilon\in\bS^{d-1}_+$. It follows from Lemma \ref{LogLipschitz} that there exists $C>1$ such that
$$F_{R,\eta,r}\big(\Gamma'(1,\bxi)M_0k(\bupsilon)\Phi_t\big)\leq C^{(d-1)\eta}  F_{C^{-(d-1)}R,\eta,Cr}(\Gamma'(1,\bxi)M_0\widetilde{n}(\by(\bupsilon))\Phi_t)$$
for any $\bupsilon\in\bS^{d-1}_+$. As $\by$ is smooth, \eqref{nonescape1+} follows from Proposition \ref{mainprop}.
\end{proof}

\section{The main lemma}\label{mainlemma}
\begin{lem}
Under the assumptions of Theorem \ref{mainthm},
\eqlabel{mainlem}{\lim_{K\to\infty}\limsup_{T\to\infty}\left|\bM_\lambda(T,\underline{z})-\bM_\lambda^{(K)}(T,\underline{z})\right|=0.}
\end{lem}
\begin{proof}
We have
$$\left|\bM_\lambda(T,\underline{z})-\bM_\lambda^{(K)}(T,\underline{z})\right|\leq\int_{\cN_{c,T}(\sig^*,\bupsilon)\ge K}(\cN_{c,T}(\sig^*,\bupsilon)+1)^{\operatorname{Re}(z_1)+\cdots+\operatorname{Re}(z_m)}\lambda(d\bupsilon)$$
where $\sig^*=\displaystyle\max_{1\leq j\leq m}\sig_j$.   We now split the integral over the upper and lower hemispheres. The integral over the upper hemisphere $\bS^{d-1}_+$ vanishes in the limit in view of Proposition \ref{mainprop'} and the upper bounds \eqref{geoest+} and \eqref{NFrel}. The analogous statement for $\bS^{d-1}_-$ follows by symmetry, since the quantity $\cN_{c,T}(\sig^*,\bupsilon)$ for a given $\bxi$ has the same value as $\cN_{c,T}(\sig^*,-\bupsilon)$ for $-\bxi$, with everything else (including $M_0$) being fixed.
\end{proof}

This completes the proof of Theorem \ref{mainthm}.

\section{Proof of Corollary \ref{cor1} and Corollary \ref{cor2}}\label{secProof}
\subsection{Proof of Corollary \ref{cor1}}

%
%
%
%
%
%

We will need the following variant of Siegel's formula.

\begin{prop}\label{Sievar}\cite[Proposition 14]{EMV15}
If $F\in L^1(\bR^d\times\bR^d)$, then
\eq{\int_{\Gamma'\backslash G'}\sum_{\bm_1\neq\bm_2\in\bZ^d}F(\bm_1 g,\bm_2 g)\, dm_{\Gamma'\backslash G'}(g)=\int_{\bR^d\times\bR^d}F(\bx,\by)\,d\bx\,d\by.}
\end{prop}

Here (and below) we view $g=\iota(\Gamma' g)\in G'$ as the representative of the coset $\Gamma' g$ in the fundamental domain $\cF'$.

For the proof of Corollary \ref{cor1}, we first prove
\eqlabel{mixexp}{\displaystyle\sum_{\underline{r}=(r_1,r_2)\in\bZ^2_{\ge0}}r_1r_2E_{c}(\underline{r},\underline{\sig})=\sig_1\sig_2+\min\set{\sig_1,\sig_2},}
 for $\underline{\sig}=(\sig_1,\sig_2)$. To deduce \eqref{mixexp}, note that
$$\displaystyle\sum_{(r_1,r_2)\in\bZ^2_{\ge0}}r_1r_2E_{c}(\underline{r},\underline{\sig})=\int_{\Gamma'\backslash G'}\sum_{\bm_1,\bm_2\in\bZ^d}\chi_{\mathfrak{C}_c(\sig_1)}(\bm_1 g)\chi_{\mathfrak{C}_c(\sig_2)}(\bm_2 g)\, dm_{\Gamma'\backslash G'}(g).$$
Recall that for $\sig>0$, the area of $\mathfrak{C}_c(\sig)$ is precisely $\sig$. Applying Proposition \ref{Sievar}, the off-diagonal part of the right-hand side is equal to
$$\int_{\bR^d\times\bR^d}\chi_{\mathfrak{C}_c(\sig_1)}(\bx)\chi_{\mathfrak{C}_c(\sig_2)}(\by)\,d\bx\,d\by=\sig_1\sig_2.$$
For the diagonal part, let $\mathfrak{C}=\mathfrak{C}_c(\sig_1)\cap\mathfrak{C}_c(\sig_2)$ and note that $\mathfrak{C}=\mathfrak{C}_c(\min\set{\sig_1,\sig_2})$. Then the diagonal part is evaluated as follows:
\eq{\begin{aligned}
& \int_{\Gamma'\backslash G'}\sum_{\bm\in\bZ^d} \chi_\mathfrak{C}(\bm g)\, dm_{\Gamma'\backslash G'}(g) \\
&=\int_{\Gamma\backslash G}\sum_{\bm\in\bZ^d}\int_{\bZ^d\backslash\bR^d}\chi_\mathfrak{C}((\bm+\bxi)M)\,d\bxi\, dm_{\Gamma\backslash G}(M)\\
&=\int_{\Gamma\backslash G}\int_{\bR^d}\chi_\mathfrak{C}(\bx M)\,d\bx\,dm_{\Gamma\backslash G}(M)=\int_{\bR^d}\chi_\mathfrak{C}(\bx)\,dx=\min\set{\sig_1,\sig_2}.
\end{aligned}}
This completes the proof of \eqref{mixexp}.

For the proof of Corollary \ref{cor1}, observe that for $z_1=z_2=1$ one of the hypotheses of Theorem \ref{mainthm} holds under the assumption of Corollary \ref{cor1}. Combining \eqref{mixexp} with the property (b) above, Corollary \ref{cor1} then follows from Theorem \ref{mainthm}.

\subsection{Proof of Corollary \ref{cor2}}
In this section we show that Corollary \ref{cor1} implies Corollary \ref{cor2}. Throughout this section we assume that the statement of Corollary \ref{cor1} holds.

\begin{lem}\label{pclem'}
Let $h\in \operatorname{C}(\bS^{d-1})$ and $\sig_1,\sig_2>0$. Then
\eqlabel{pclem}{\begin{aligned}\lim_{N\to\infty}\int_{\bS^{d-1}}\sum_{\substack{j_1,j_2=1 \\ j_1\neq j_2}}^N\chi_{[0,\sig_1]}(N^{\frac{1}{d-1}}\bd(\balpha_{j_1},\balpha))&\chi_{[0,\sig_2]}(N^{\frac{1}{d-1}}\bd(\balpha_{j_2},\balpha))h(\balpha)d\balpha\\&=\sig_1\sig_2\int_{\bS^{d-1}}h(\balpha)d\balpha.\end{aligned}}
\end{lem}
\begin{proof}
By Corollary \ref{cor1}, the left-hand side of \eqref{pclem} without the restriction $j_1\neq j_2$ converges to
\eqlabel{offdiag}{\min\set{\sig_1,\sig_2}\int_{\bS^{d-1}}h(\balpha)d\balpha+\sig_1\sig_2\int_{\bS^{d-1}}h(\balpha)d\balpha.}
On the other hand, the diagonal part $j_1=j_2$ of the left-hand side of \eqref{pclem} is
\eqlabel{diag1}{\int_{\bS^{d-1}}\sum_{j=1}^N\chi_{[0,\min\set{\sig_1,\sig_2}]}(N^{\frac{1}{d-1}}\bd(\balpha_{j},\balpha))h(\balpha)d\balpha.}
Since $h$ is continuous and $\bd(\balpha_j,\balpha)\ll_{\sig_1,\sig_2} N^{-\frac{1}{d-1}}$ for any $1\leq j\leq N$, for any $\eps>0$ there exists $N_0$ such that for all $N\ge N_0$ we have $|h(\balpha_j)-h(\balpha)|<\eps$ for any $\balpha\in\bS^{d-1}$ and $1\leq j\leq N$. It follows that the integral \eqref{diag1} is approximated by
\eqlabel{diag2}{\int_{\bS^{d-1}}\sum_{j=1}^N\chi_{[0,\min\set{\sig_1,\sig_2}]}(N^{\frac{1}{d-1}}\bd(\balpha_{j},\balpha))h(\balpha_j)d\balpha}
up to error $<\eps\min\set{\sig_1,\sig_2}$. Hence, \eqref{diag1} converges to 
\eqlabel{diag3}{\begin{aligned}
&\lim_{N\to\infty}\int_{\bS^{d-1}}\sum_{j=1}^N\chi_{[0,\min\set{\sig_1,\sig_2}]}(N^{\frac{1}{d-1}}\bd(\balpha_{j},\balpha))h(\balpha)d\balpha\\&=\lim_{N\to\infty}\min\set{\sig_1,\sig_2}N^{-1}\sum_{j=1}^N h(\balpha_j)\\
&=\min\set{\sig_1,\sig_2}\int_{\bS^{d-1}}h(\balpha)d\balpha ,
\end{aligned}}
since the $\balpha_j$'s are uniformly distributed over $\bS^{d-1}$. Therefore, the second summand of \eqref{offdiag} is the off-diagonal contribution appearing in \eqref{pclem} as desired.
\end{proof}

\begin{lem}\label{pclem2'}
Let $g\in \operatorname{C}(\bS^{d-1}\times\bS^{d-1})$ and $\sig_1,\sig_2>0$. Then
\eqlabel{pclem2}{\begin{aligned}\lim_{N\to\infty}\sum_{\substack{j_1,j_2=1\\ j_1\neq j_2}}^Ng(\balpha_{j_1},\balpha_{j_2})\int_{\bS^{d-1}}\chi_{[0,\sig_1]}(N^{\frac{1}{d-1}}\bd(\balpha_{j_1},\balpha))&\chi_{[0,\sig_2]}(N^{\frac{1}{d-1}}\bd(\balpha_{j_2},\balpha))d\balpha\\&=\sig_1\sig_2\int_{\bS^{d-1}}g(\balpha,\balpha)d\balpha.\end{aligned}}
\end{lem}
\begin{proof}
Since $g$ is continuous and $\bd(\balpha_j,\balpha)\ll_{\sig_1,\sig_2}N^{-\frac{1}{d-1}}$ for any $1\leq j\leq N$, for any $\eps>0$ there exists $N_0$ such that for all $N\geq N_0$ we have $|g(\balpha,\balpha)-g(\balpha_{j_1},\balpha_{j_2})|<\eps$ for any $\balpha\in\bS^{d-1}$ and $1\leq j_1,j_2\leq N$. It follows that the left-hand side of \eqref{pclem2} is approximated by
\eqlabel{approx}{\sum_{\substack{j_1,j_2=1\\ j_1\neq j_2}}^N\int_{\bS^{d-1}}g(\balpha,\balpha)\chi_{[0,\sig_1]}(N^{\frac{1}{d-1}}\bd(\balpha_{j_1},\balpha))\chi_{[0,\sig_2]}(N^{\frac{1}{d-1}}\bd(\balpha_{j_2},\balpha))d\balpha}
up to error
\eqlabel{error}{<\eps\int_{\bS^{d-1}}\sum_{\substack{j_1,j_2=1\\ j_1\neq j_2}}^N\chi_{[0,\sig_1]}(N^{\frac{1}{d-1}}\bd(\balpha_{j_1},\balpha))\chi_{[0,\sig_2]}(N^{\frac{1}{d-1}}\bd(\balpha_{j_2},\balpha))d\balpha\ll\eps\sig_1\sig_2,}
where the last inequality follows from Lemma \ref{pclem'} with the choice $h=1$. Applying Lemma \ref{pclem'} again for \eqref{approx} with the choice $h(\balpha)=g(\balpha,\balpha)$, we conclude the proof.
\end{proof}

Corollary \ref{cor2} now follows from Lemma \ref{pclem2'} by approximating $f\in \operatorname{C}_0(\bS^{d-1}\times\bS^{d-1}\times\bR)$ from above and below by finite linear combinations of functions of the form
\eqlabel{lincom}{\widetilde{f}(\bx,\by,z)=g(\bx,\by)\int_{\bR^{d-1}}\chi_{\cB_{\sigma_1}^{d-1}}(\bw)\chi_{\cB_{\sigma_2}^{d-1}}(\bw+z\be_1)\,d\bw,}
for suitable choices of $g\in \operatorname{C}(\bS^{d-1}\times\bS^{d-1})$ and $\sig_1,\sig_2>0$.

\appendix

\section{Brjuno type condition}\label{sec: Brjuno}

Following \cite{LDG19} (cf. also \cite{BF19}) we say that $\bxi\in\bR^{d}$ is a $s$-Brjuno vector if
\eq{\displaystyle\sum_{n=0}^{\infty}2^{-\frac{n}{s}}\displaystyle\displaystyle\max_{\substack{\bm\in\bZ^d\setminus\set{0}\\ 0<|\bm|\leq 2^{n}}}\log\frac{1}{|\bxi\cdot\bm|_\bZ}<\infty.}
The classical Brjuno condition corresponds to $s=1$.

In this section, we prove that for $s>\frac{\rho+1}{\nu}$, every $s$-Brjuno vector is $(\rho,0,\nu)$-vaguely Diophantine.
Given $\bxi\in\bR^{d}$ let us define $\phi:\bN\to\bR^{>0}$ by
\eq{\phi(N):=\max_{\substack{\bm\in\bZ^d\setminus\set{0}\\ 0<|\bm|\leq N}}\log\frac{1}{|\bxi\cdot\bm|_\bZ}.}
Then the definition of $\zeta(\bxi,T)$ can be written in terms of $\phi(N)$ as follows:
\eq{\begin{aligned}\zeta(\bxi,T)&=\min\set{N\in\bN: e^{-\phi(N)}\leq T^{-1}}\\&=\min\set{N\in\bN: \phi(N)\ge\log T}.\end{aligned}}

Suppose that $\bxi$ is $s$-Brjuno type for some $s>\frac{\rho+1}{\nu}$. Then we have $\sum_{n}2^{-\frac{n}{s}}\phi(2^{n})<\infty$, hence $\phi(t)\leq t^{\frac{1}{s}}\log 2$ for sufficiently large $t$. It follows that
\eq{\zeta(\bxi,2^{l-1})=\min\set{N\in\bN: \phi(N)\ge (l-1)\log2}\geq (l-1)^{s}}
for sufficiently large $l$. Thus, it implies that $\bxi$ is $(\rho,0,\nu)$-vaguely Diophantine since
\eq{\displaystyle\sum_{l=1}^{\infty}l^\rho\, \zeta(\bxi,2^{l-1})^{-\nu}\ll \displaystyle\sum_{l=1}^{\infty}l^{\rho-s\nu}<\infty.}

\section{Counterexamples}\label{sec: counterexamples}

The following theorem establishes that for $d=2$ there exist uncountably many $\bxi\in\bR^2\setminus\bQ^2$ such that the pair correlation function of the directions of the affine lattice $\bZ^2+\bxi$ diverges, whilst converging to the Poisson limit along a subsequence. 

\begin{thm}
    Fix $s>0$. Then there exists a set ${\mathcal C}\subset\bR^2$ of second Baire category with the property that for $\bxi\in{\mathcal C}$ there are sequences $(M_j)_{j\in\bN}$, $(N_j)_{j\in\bN}$ with $M_j,N_j\to\infty$ such that
\begin{equation}
R^2_{M_j}(s) \geq \frac{\log M_j}{\log\log\log M_j}
\end{equation}
and 
\begin{equation}
\lim_{j\to \infty }R^2_{N_j}(s)=\frac{\pi^{\frac{d-1}{2}} s^{d-1}}{\Gamma(\tfrac{d+1}{2})} .
\end{equation}
\end{thm}

The proof of this statement follows from Theorem \ref{thm000} and the next lemma, by a Baire category argument identical to that used in \cite{Sa97}; see also \cite[\S9]{M03}.

\begin{lem}
    Let $\bxi\in\bQ^2$. Then there exists a constant $C_{\bxi}$ such that for any $s>0$ and $T>0$ we have
\begin{equation}
R^2_{N_0(T)}(s) \geq C_{\bxi} \log N_0(T) .
\end{equation}
\end{lem}
\begin{proof}
We first observe that for any $s>0$ and $N\in\bN$
$$R^2_N(s) \geq \frac{1}{N} \#\left\{ (j_1,j_2) \,:\, j_1,j_2\leq N, \, j_1\neq j_2,\, \bupsilon_{j_1}=\bupsilon_{j_2} \right\}.$$
Let us first consider the case $\bxi=(0,0)$. We denote the set of primitive lattice points by $\bZ^2_{\operatorname{prim}}$. Then 
    \eq{\begin{aligned}
        R^2_{N_0(T)}(s) &\geq \frac{1}{N_0(T)} \sum_{\bv\in\bZ^2_{\operatorname{prim}}}\sum_{\substack{k,l\neq0,\; k\neq l,\\\|k\bv\|,\|l\bv\|\leq T}}1
        \\&\gg \frac{1}{T^2}\sum_{\substack{0\leq m<n\leq \frac{T}{2},\\ \operatorname{gcd}(m,n)=1}}\left(\frac{T}{n}\right)\left(\frac{T}{n}-1\right)
        \\&\gg \sum_{1\leq n\leq \frac{T}{2}}\frac{\varphi(n)}{n^2}
        \gg \log T \gg \log N_0(T),
    \end{aligned}}
where $\varphi$ denotes Euler's totient function. The above lower bound for the sum over the totient functions follows (via diadic decomposition) from the classic asymptotics
$\displaystyle\sum_{1\leq n\leq x} \varphi(n) \sim \frac{3}{\pi^2} x^2$.

The case $\bxi\in\bQ^2\setminus\set{0}$ is similar to the above. Let $\bxi=\frac{1}{q}\bp$ for some $\bp=(p_1,p_2)\in\bZ^2\setminus\set{0}$ and $q\in\bN$. Then we have
\eq{\begin{aligned}
        R^2_{N_0(T)}(s) &\geq \frac{1}{N_0(T)} \sum_{\bv\in\bZ^2_{\operatorname{prim}}}\sum_{\substack{k,l\neq 0,\;k\neq l,\\ \|k\bv\|\leq qT,\; k\bv\equiv \bp\; (\operatorname{mod} \;q),\\ \|l\bv\|\leq qT,\;l\bv\equiv \bp\; (\operatorname{mod} \;q)}}1
        \\&\gg \frac{1}{T^2}\sum_{\substack{0\leq m<n\leq \frac{T}{2q},\\ \operatorname{gcd}(m,n)=1}}\left(\frac{T}{n}\right)\left(\frac{T}{n}-1\right)\\&\gg \sum_{1\leq n\leq \frac{T}{2q}}\frac{\varphi(n)}{n^2} \gg \log T \gg \log N_0(T).
    \end{aligned}}
This completes the proof of the lemma.
\end{proof}

\def\cprime{$'$} \def\cprime{$'$} \def\cprime{$'$}
\providecommand{\bysame}{\leavevmode\hbox to3em{\hrulefill}\thinspace}
\providecommand{\MR}{\relax\ifhmode\unskip\space\fi MR }
\providecommand{\MRhref}[2]{%
  \href{http://www.ams.org/mathscinet-getitem?mr=#1}{#2}
}
\providecommand{\href}[2]{#2}

\end{document}